\documentclass[11pt]{article}%
\usepackage{amsfonts}
\usepackage{amsmath}
\usepackage{geometry}
\usepackage{amssymb}
\usepackage{graphicx}%
\providecommand{\U}[1]{\protect\rule{.1in}{.1in}}

\newtheorem{theorem}{Theorem}

\newtheorem{definition}[theorem]{Definition}

\newtheorem{lemma}[theorem]{Lemma}

\newtheorem{proposition}[theorem]{Proposition}
\newtheorem{remark}[theorem]{Remark}

\newenvironment{proof}[1][Proof]{\noindent\textbf{#1.} }{\ \rule{0.5em}{0.5em}}
\begin{document}

\title{ Well-posedness of the vector advection equations by stochastic perturbation. }
\author{Franco Flandoli\thanks{Dipartimento di Matematica, Universit\`{a} di Pisa,
Italy, E-mail: \textsl{flandoli@dma.unipi.it}}, Christian
Olivera\thanks{Departamento de Matem\'{a}tica, Universidade Estadual de
Campinas, Brazil. E-mail: \textsl{colivera@ime.unicamp.br}. }}
\date{}
\maketitle

\begin{abstract}
A linear stochastic vector advection equation is considered. The equation may
model a passive magnetic field in a random fluid. The driving velocity field
is a integrable to a certain power and the noise is infinite dimensional. We
prove that, thanks to the noise, the equation is well posed in a suitable
sense, opposite to what may happen without noise.

\end{abstract}



\section{Introduction}

\label{Intro}

Consider the linear stochastic vector advection equation in the unknown random
field $B:\Omega\times\lbrack0,T]\times\mathbb{R}^{d}\rightarrow\mathbb{R}^{d}$%

\begin{equation}
dB+\left(  v\cdot\nabla B-B\cdot\nabla v\right)  \ dt+\sum_{k=1}^{\infty
}\left(  \sigma_{k}\cdot\nabla B-B\cdot\nabla\sigma_{k}\right)  \circ
dW_{t}^{k}=0 \label{Adve}%
\end{equation}
where $v:[0,T]\times\mathbb{R}^{d}\rightarrow\mathbb{R}^{d}$ and $\sigma
_{k}:\mathbb{R}^{d}\rightarrow\mathbb{R}^{d}$, $k\in\mathbb{N}$, are given
divergence free vector fields and $\left(  W_{\cdot}^{k}\right)
_{k\in\mathbb{N}}$ is a family of independent real-valued Brownian motions on
the filtered probability space $\left(  \Omega,\mathcal{F},\mathcal{F}%
_{t},P\right)  $. We write the problem in a generic dimension $d\geq1$ but our
investigation is strongly motivated by the case $d=3$. The stochastic
integration is to be understood in the Stratonovich sense. This equation may
model a passive vector field $B$, like a magnetic field, in a turbulent fluid
with a non-trivial average component $v$ and random component $\sum
_{k=1}^{\infty}\sigma_{k}dW_{t}^{k}$. The general structure of the noise
assumed here is inspired by the theory of diffusion of passive scalars and
vector fields in turbulent fluids, see for instance \cite{FaGaVe} and is also
motivated by the recent proposal for a variational principle approach to fluid
mechanics, see \cite{Holm} (although the equations in \cite{Holm} are always
nonlinear, with random $v$ influenced by $B$, hence more difficult than those
studied here). Particular cases of this equation have been considered before
in \cite{FMau}, \cite{Fla}; see also \cite{Brez} and references therein; but
the generality assumed here is important from the physical viewpoint and the
proofs are new. We impose below some simplifying assumptions on the vector
fields $\sigma_{k}$.

We aim at studying existence and uniqueness, under low regularity assumption
on $v$. More precisely, we assume that%

\begin{equation}
v\in L^{\infty}\left(  [0,T],L^{p}(\mathbb{R}^{d})\right)  \text{\qquad for
some}\ p>d.\label{condition on v}%
\end{equation}
For sake of simplicity we also assume $p\geq2$. Under this condition,
existence and uniqueness is not a classical result:\ indeed, in the
deterministic case (all $\sigma_{k}=0$), it is not true, as shown in
\cite{FMau} and \cite{Fla}. Thus the result of existence and uniqueness is due
to the random perturbation. The same question was considered in \cite{FMau}
under an H\"{o}lder condition on $v$ and a partial result is given in
\cite{Fla} for $v$ having suitable integrability. But in both cases the noise
was the standard Brownian motion in $\mathbb{R}^{d}$, without a space
structure. The main novelties of the present work with respect to \cite{FMau},
\cite{Fla} are:\ i) the approach, based on the new concept of quasi-regular
solution, recently introduced in \cite{Fre2} for transport type equations,
approach which allows one to prove certain properties in a much easier
way;\ ii) the noise is much more general and in line with the physical and
geometrical literature, \cite{FaGaVe}, \cite{Holm}; iii) the proof for $v$
with only integrability properties (instead of H\"{o}lder continuity) is here
complete, w.r.t. \cite{Fla} which gave only general arguments, also due to the
more synthetic approach used here. The main restriction, compared to other
works on this subject, is the notion of uniqueness used here: it is uniqueness
in the class of processes adapted to the filtration generated by the Brownian
motions. This is more restrictive than pathwise uniqueness, see Remark
\ref{remark unique}.

About condition (\ref{condition on v}), let us add some historical remarks.
Our source of inspiration is the paper \cite{Krylov}, where the authors proved
existence and uniqueness of strong solutions for the SDE%

\begin{equation}
X_{s,t}(x)=x+\int_{s}^{t}v(r,X_{s,r}(x))\ dr+W_{t}-W_{s}\,.\label{itoass}%
\end{equation}
This is the equation of characteristics associated to the SPDE (\ref{Adve}) in
the particular case when the noise is just the standard Brownian motion in
$\mathbb{R}^{d}$, without a space structure. A condition similar to
(\ref{condition on v}) was also considered in \cite{Beck}, \cite{Fre1} and
\cite{NO} to study scalar problems like linear transport equations and linear
continuity equations. In fluid mechanics, in the viscous case of Navier-Stokes
equations, when such condition holds for a\ weak solution, then such solution
is unique and more regular (it is a particular case of the so called
Lady\v{z}enskaja-Prodi-Serrin condition). Here the framework is of course
different: $v$ is given, not the unknown, and the equation is inviscid; hence
a true comparison is not possible. We only stress some parallelism between
these theories.

In Section \ref{sect prel def res}, after some necessary preliminaries which
include the It\^{o} formulation of equation (\ref{Adve}), the concept of
stochastic exponentials and the assumptions on the noise, we define the notion
of quasiregular weak solution and formulate our main existence and uniqueness
results. Then, after some technical results given in Section \ref{sect prelim}%
, we prove existence of solutions in Section \ref{section existence} and
uniqueness in Section \ref{sect unique}.

\section{Main results\label{sect prel def res}}

This section is devoted to the definition of quasiregular weak solution and to
the statement of our main results of existence and uniqueness, see Section
\ref{sect def sol}.

For this purpose, we have to introduce a few preliminaries. Section
\ref{subsect Ito} gives the It\^{o} formulation of equation (\ref{Adve}),
where the second order operator $\mathcal{L}$ given by (\ref{def L}) arises;
the definition of solution and the whole rigorous analysis is based on the
It\^{o} form. Section \ref{subsect exponent} describes the concept of
stochastic exponential used in the definition of solution and the the
filtration associated to the Brownian motions used in the statement of
uniqueness of solutions. Finally, Section \ref{subsect noise} presents the
assumptions on the noise. 

\subsection{It\^{o} formulation\label{subsect Ito}}

It is convenient to introduce the notation of the Lie bracket between vector
fields%
\[
\left[  A,B\right]  =A\cdot\nabla B-B\cdot\nabla A
\]
which is also equal to the Lie derivative $\mathcal{L}_{A}B$ and also, for
divergence free fields, to $\operatorname{curl}\left(  A\wedge B\right)  $. In
Stratonovich form equation (\ref{Adve}) then reads%
\[
dB+\left[  v,B\right]  \ dt+\sum_{k=1}^{\infty}\left[  \sigma_{k},B\right]
\circ dW_{t}^{k}=0.
\]
Its It\^{o} formulation is%
\begin{equation}
dB+\left[  v,B\right]  \ dt+\sum_{k=1}^{\infty}\left[  \sigma_{k},B\right]
\ dW_{t}^{k}=\frac{1}{2}\sum_{k=1}^{\infty}\left[  \sigma_{k},\left[
\sigma_{k},B\right]  \right]  \ dt.\label{Adve Ito}%
\end{equation}
Before we justify the claim we have to clarify that we wrote the Stratonovich
formulation above in a formal way, for the purpose of a better physical
understanding, but at the rigorous level we shall always use its It\^{o}
formulation. For this reason, we do not provide a rigorous proof of the
equivalence of the two formulations but only a formal argument. Then, in the
next section, we give a rigorous definition of solution of the It\^{o}
equation only.

Let us show that (\ref{Adve}) leads to (\ref{Adve Ito}). Recall that
Stratonovich integral differs from It\^{o} integral by 1/2 mutual variation:
$X\circ dW=XdW+\frac{1}{2}d\left\langle X,W\right\rangle $; where, in the case
of interest to us when $X$ is vector valued and $W$ is real valued, by
$\left\langle X,W\right\rangle $ we mean the vector of components
$\left\langle X^{\alpha},W\right\rangle $. Then%
\[
\left[  \sigma_{k},B\right]  \circ dW_{t}^{k}=\left[  \sigma_{k},B\right]
\ dW_{t}^{k}+d\left\langle \left[  \sigma_{k},B\right]  ,W^{k}\right\rangle
_{t}.
\]
Now
\[
d\left\langle \left[  \sigma_{k},B\right]  ,W^{k}\right\rangle _{t}=\left(
\sigma_{k}\cdot\nabla\right)  d\left\langle B,W^{k}\right\rangle
_{t}-d\left\langle B,W^{k}\right\rangle _{t}\cdot\nabla\sigma_{k}.
\]
From the equation for $dB$ and the property that the mutual variations between
$W^{k}$ and BV functions or stochastic integrals with respect to $W^{j}$ for
$j\neq k$ are zero (and $\,d\left\langle W^{k},W^{k}\right\rangle _{t}=dt$) we
get%
\[
d\left\langle B,W^{k}\right\rangle _{t}=d\left\langle \int_{0}^{\cdot}\left[
\sigma_{k},B_{s}\right]  \ dW_{s}^{k},W^{k}\right\rangle _{t}=\left[
\sigma_{k},B_{t}\right]  \ dt.
\]
Therefore we deduce (formally speaking) (\ref{Adve Ito}).

We have introduced the second order differential operator, acting on smooth
vector fields $B$, defined as%
\begin{equation}
\mathcal{L}B\left(  x\right)  :=\frac{1}{2}\sum_{k=1}^{\infty}\left[
\sigma_{k},\left[  \sigma_{k},B\right]  \right]  \left(  x\right)
.\label{def L}%
\end{equation}
We shall see in Proposition \ref{Prop ellipticity} that this operator is well
defined and strongly elliptic, under the assumptions on the noise imposed below.

\subsection{Stochastic exponentials\label{subsect exponent}}

Let $\left(  \Omega,\mathcal{F},\mathcal{F}_{t},P\right)  $ be the filtered
probability space introduced above, with the sequence $\left\{  W_{t}%
^{k}\right\}  _{k\in\mathbb{N}}$ of independent Brownian motions. Let
$\mathcal{G}_{t}$ be the associated filtration:\
\[
\mathcal{G}_{t}=\sigma\left\{  B_{s}^{k};s\in\left[  0,t\right]
,k\in\mathbb{N}\right\}  .
\]
Let $\overline{\mathcal{G}}_{t}$ be the completed filtration. For some $T>0$,
let%
\[
\mathcal{H}=L^{2}\left(  \Omega,\overline{\mathcal{G}}_{T},P\right)
\]%
\[
F=\cup_{n\in\mathbb{N}}L^{2}\left(  0,T;\mathbb{R}^{n}\right)
\]%
\[
\mathcal{D}=\left\{  e_{f}\left(  T\right)  ;f\in F\right\}
\]
where, for $n\in\mathbb{N}$, $f\in L^{2}\left(  0,T;\mathbb{R}^{n}\right)  $,
with components $f_{1},...,f_{n}$, we set%
\[
e_{f}\left(  t\right)  =\exp\left(  \sum_{k=1}^{n}\int_{0}^{t}f_{k}\left(
s\right)  dW_{s}^{k}-\frac{1}{2}\sum_{k=1}^{n}\int_{0}^{t}\left\vert
f_{k}\left(  s\right)  \right\vert ^{2}ds\right)
\]
for $t\in\left[  0,T\right]  $. From It\^{o} formula%
\[
de_{f}\left(  t\right)  =\sum_{k=1}^{n}f_{k}\left(  t\right)  e_{f}\left(
t\right)  dW_{t}^{k}.
\]
The following result is known, see the argument in \cite{Nualart}:

\begin{lemma}
\label{lemma Wiener space}$\mathcal{D}$ is dense in $\mathcal{H}$.
\end{lemma}

\subsection{Structure and assumptions on the noise\label{subsect noise}}

Let $\left(  \sigma_{k}\right)  _{k\in\mathbb{N}}$ be a sequence of twice
differentiable divergence free vector fields:%
\begin{equation}
\sigma_{k}\in C^{2}\left(  \mathbb{R}^{d},\mathbb{R}^{d}\right)
,\qquad\operatorname{div}\sigma_{k}=0\label{ass1 on sigma}%
\end{equation}
such that
\begin{equation}
\sum_{k=1}^{\infty}\left\vert \sigma_{k}\left(  x\right)  \right\vert
^{2}<\infty\label{ass2 on sigma}%
\end{equation}
for every $x\in\mathbb{R}^{d}$. The matrix-valued function $Q\left(
x,y\right)  \in\mathbb{R}^{d\times d}$, $x,y\in\mathbb{R}^{d}$, given by
\[
Q^{\alpha\beta}\left(  x,y\right)  :=\sum_{k=1}^{\infty}\sigma_{k}^{\alpha
}\left(  x\right)  \sigma_{k}^{\beta}\left(  y\right)
\]
is well defined, (we write $Q^{\alpha\beta}\left(  x,y\right)  $,
$\alpha,\beta=1,...,d$ for its components and similarly for $\sigma
_{k}^{\alpha}\left(  x\right)  $). Our main assumptions on the noise are:
$Q\left(  x,y\right)  $ is twice continuously differentiable in $\left(
x,y\right)  $, bounded with bounded first and second derivatives, that we
summarize in the notation%
\begin{equation}
Q\in C_{b}^{2}\label{assumpt on Q}%
\end{equation}
and%
\begin{equation}
Q\left(  x,x\right)  \geq\nu Id_{\mathbb{R}^{d}}\label{ellipt}%
\end{equation}
for some $\nu>0$, uniformly in $x\in\mathbb{R}^{d}$. It will be shown below,
in Proposition \ref{Prop ellipticity}, that $Q\left(  x,x\right)  $ appears in
the principal part of the second order differential operator $\mathcal{L}$
given by (\ref{def L}); condition (\ref{ellipt}) implies strong ellipticity of
this operator. 

\begin{remark}
In the literature it often assumed that there exists a matrix-valued function
$Q\left(  x\right)  \in\mathbb{R}^{d\times d}$, $x\in\mathbb{R}^{d}$, such
that
\[
Q\left(  x,y\right)  =Q\left(  x-y\right)
\]
(this is equivalent to assume that the Gaussian random field $\sum
_{k=1}^{\infty}\sigma_{k}\left(  x\right)  W_{t}^{k}$ is space homogeneous).
The value $Q\left(  0\right)  =Q\left(  x,x\right)  $ plays a special role and
is often assumed to be a non-degenerate matrix, for simplicity%
\[
Q\left(  0\right)  =Id
\]
the identity matrix in $\mathbb{R}^{d}$. We do not impose these additional
conditions but only (\ref{ellipt}) which corresponds to the non degeneracy of
$Q\left(  0\right)  $.
\end{remark}

\subsection{Definition of solution and main result\label{sect def sol}}

We present now the setting and a suitable definition of quasiregular weak
solutions to equation (\ref{Adve}) , adapted to treat the problem of
well-posedness. Throughout the paper we assume that the vector field $v$
satisfies
\begin{equation}
v\in L^{\infty}\left(  [0,T],L^{p}(\mathbb{R}^{d};\mathbb{R}^{d})\right)
\qquad\text{for some }p\text{ such that }p>d,\text{ }p\geq2\label{con1}%
\end{equation}
\begin{equation}
\operatorname{div}v(t,x)=0\,.\label{con2}%
\end{equation}

\begin{remark}
The condition $p\geq2$, that we do not consider restrictive because we have in
mind mainly the case $d=3$, is imposed to treat solutions $B$ of $L^{2}$
class, but this could be generalized. Also the restriction of diverge free
fields is imposed to simplify a number of arguments and does not look so
restrictive having in mind fluid dynamics; to generalize it, one should
require suitable integrability of the distributional divergences, a technical
generalization that we omit. 
\end{remark}

Moreover, the initial condition is taken to be \
\begin{equation}
B_{0}\in L^{4}(\mathbb{R}^{d})\cap L^{\infty}(\mathbb{R}^{d})\,,\qquad
\operatorname{div}B_{0}=0.\label{assumpt on B0}%
\end{equation}
The noise, as just said above, satisfies (\ref{ass1 on sigma}),
(\ref{ass2 on sigma}), (\ref{assumpt on Q}), (\ref{ellipt}) (the latter
assumption is not needed to give the definition of solution; but it is used in
the proof of both existence and uniqueness of solutions).

The next definition tells us in which sense a stochastic process is a
quasiregular weak solution of (\ref{Adve}). We formally use the identity
\begin{align*}
\int\left[  A,B\right]  \cdot C\ dx  & =\int\left(  \left(  A\cdot
\nabla\right)  B-\left(  B\cdot\nabla\right)  A\right)  \cdot C\ dx\\
& =-\int B\cdot\left(  A\cdot\nabla\right)  C-A\cdot\left(  B\cdot
\nabla\right)  C\ dx
\end{align*}
which holds true for sufficiently smooth and integrable fields such that
$\operatorname{div}A=\operatorname{div}B=0$. Moreover, we use the adjoint
operator $\mathcal{L}^{\ast}$, defined in Proposition \ref{Prop ellipticity}%
\ below, which maps test functions $\varphi\in C_{c}^{\infty}(\mathbb{R}%
^{d},\mathbb{R}^{d})$ into bounded continuous compact support vector fields
$\mathcal{L}^{\ast}\varphi\left(  x\right)  $. 

\begin{definition}
\label{defisolu} A stochastic process $B:[0,T]\times\mathbb{R}^{d}%
\rightarrow\mathbb{R}^{d}$, $B\in L^{2}\left(  \Omega\times\lbrack
0,T],L_{loc}^{2}(\mathbb{R}^{d}\right)  )$ is called a quasiregular weak
solution of the Cauchy problem (\ref{Adve}) when:

i) $\operatorname{div}B\left(  \omega,t\right)  =0$, in the sense of
distributions, for a.e. $\left(  \omega,t\right)  \in\Omega\times\lbrack0,T]$

ii) for any $\varphi\in C_{c}^{\infty}(\mathbb{R}^{d},\mathbb{R}^{d})$, the
real valued process $\int B(t,x)\cdot\varphi(x)dx$ has a continuous
modification which is an $\mathcal{F}_{t}$-semimartingale,

iii) for any $\phi\in C_{c}^{\infty}(\mathbb{R}^{d},\mathbb{R}^{d})$ and for
all $t\in\lbrack0,T]$, we have $\mathbb{P}$-almost surely%
\begin{align*}
&  \int B\left(  t,x\right)  \cdot\phi\left(  x\right)  \ dx-\int_{0}^{t}%
\int\left(  B\left(  s,x\right)  \cdot\left(  v\left(  s,x\right)  \cdot
\nabla\right)  \phi\left(  x\right)  -v\left(  s,x\right)  \cdot\left(
B\left(  s,x\right)  \cdot\nabla\right)  \phi\left(  x\right)  \right)
\ dxdt\\
&  -\sum_{k=1}^{\infty}\int_{0}^{t}\left(  \int\left(  B\left(  s,x\right)
\cdot\left(  \sigma_{k}\left(  x\right)  \cdot\nabla\right)  \phi\left(
x\right)  -\sigma_{k}\left(  x\right)  \cdot\left(  B\left(  s,x\right)
\cdot\nabla\right)  \phi\left(  x\right)  \right)  \ dx\right)  \ dW_{s}^{k}%
\end{align*}%
\begin{equation}
=\int B_{0}\left(  x\right)  \cdot\phi\left(  x\right)  \ dx+\frac{1}{2}%
\int_{0}^{t}\int\mathcal{L}^{\ast}\phi\left(  x\right)  \cdot B\left(
s,x\right)  \ dxdt\label{Ito SPDE}%
\end{equation}

iv)\ (\textit{Regularity in Mean}) For all $n\in\mathbb{N}$ and each function
$f\in L^{2}\left(  0,T;\mathbb{R}^{n}\right)  $, with components
$f_{1},...,f_{n}$, the deterministic function $V\left(  t,x\right)
:=\mathbb{E}[B\left(  t,x\right)  e_{f}\left(  t\right)  ]$ is a measurable
bounded function, which belongs to $L^{2}([0,T];H^{1}(\mathbb{R}^{d}))\cap
C([0,T];L^{2}(\mathbb{R}^{d}))$ and satisfies the parabolic equation%
\begin{equation}
\partial_{t}V+\left[  v-h,V\right]  =\mathcal{L}V\label{PDE}%
\end{equation}
in the weak sense, where $h\left(  t,x\right)  :=\sum_{k=1}^{n}f_{k}\left(
t\right)  \sigma_{k}\left(  x\right)  $.
\end{definition}

We have called \textit{quasiregular} this class of weak solutions because of
the regularity of the expected values $V\left(  t,x\right)  :=\mathbb{E}%
[B\left(  t,x\right)  e_{f}\left(  t\right)  ]$. Equation (\ref{Adve}) has an
hyperbolic nature, it cannot regularize the initial condition; but in the
average there is a regularization, on which we insist in the definition. 

Point (iv), which characterizes this definition compared to others, is imposed
specifically for uniqueness, beyond the fact that emphasizes a particular
regularity property. Thanks to (iv) the proof of uniqueness (understood in the
special sense stressed also in Remark \ref{remark unique} below) becomes
extremely easy compared to other proofs done in the literature for analogous
problems, see for instance \cite{FGP2}. The advantage is that it is sufficient
to use uniqueness tools for the deterministic parabolic equation (\ref{PDE}),
which is much more regular than (\ref{Adve}). The disadvantage is that the
expected values $V\left(  t,x\right)  :=\mathbb{E}[B\left(  t,x\right)
e_{f}\left(  t\right)  ]$ treated by equation (\ref{Adve}) characterize the
solution only with respect to the filtration $\left(  \overline{\mathcal{G}%
}_{t}\right)  $ defined in Section \ref{subsect exponent}.

Let us see the formal motivation for equation (\ref{PDE}). We apply formally
It\^{o} formula to the product of a solution with the stochastic exponential,
in equation (\ref{Adve Ito}). We get%

\begin{align*}
& d\left(  Be_{f}\right)  +\left[  v,B\right]  e_{f}\ dt+\sum_{k=1}^{\infty
}e_{f}\left[  \sigma_{k},B\right]  \ dW_{t}^{k}\\
& =\frac{1}{2}\sum_{k=1}^{\infty}e_{f}\left[  \sigma_{k},\left[  \sigma
_{k},B\right]  \right]  \ dt+\sum_{k=1}^{n}f_{k}Be_{f}dW_{t}^{k}+\sum
_{k=1}^{n}f_{k}e_{f}\left[  \sigma_{k},B\right]  \ dt.
\end{align*}
Taking expectation we obtain%

\begin{align*}
& \partial_{t}V+\left[  v,V\right]  =\frac{1}{2}\sum_{k}\left[  \sigma
_{k},\left[  \sigma_{k},V\right]  \right]  +\sum_{k=1}^{n}f_{k}\left[
\sigma_{k},V\right]  \\
& =\frac{1}{2}\sum_{k}\left[  \sigma_{k},\left[  \sigma_{k},V\right]  \right]
+\left[  \sum_{k=1}^{n}f_{k}\sigma_{k},V\right]  .
\end{align*}

Thus we obtain equation (\ref{PDE}), that we have to interpret rigorously in
weak form.

\begin{theorem}
Under assumptions (\ref{ass1 on sigma}), (\ref{ass2 on sigma}), (\ref{con1}),
(\ref{con2}), (\ref{assumpt on B0}), (\ref{assumpt on Q}), (\ref{ellipt}), a
quasiregular weak solution of the Cauchy problem (\ref{Adve}) exists.
\end{theorem}

Recall the definition of the filtration $\left(  \overline{\mathcal{G}}%
_{t}\right)  $ given in Section \ref{subsect exponent}.

\begin{theorem}
Under the assumptions of the previous theorem, let $B^{i}$, $i=1,2$, be two
quasi-regular weak solutions of equation (\ref{Adve}) with the same initial
condition $B_{0}$. Assume that $\int B^{i}(t,x)\varphi(x)dx$ is $\overline
{\mathcal{G}}_{t}$-adapted, for both $i=1,2$, for every $\varphi\in
C_{c}^{\infty}\left(  \mathbb{R}^{3},\mathbb{R}^{3}\right)  $. Then
$B^{1}=B^{2}$.
\end{theorem}

\begin{remark}
\label{remark unique}This notion of uniqueness is weaker than the classical
notion of pathwise uniqueness. One says that pathwise uniqueness holds when
given a filtered probability space $\left(  \Omega,\mathcal{F},\mathcal{F}%
_{t},P\right)  $, a sequence $\left\{  W_{t}^{k}\right\}  _{k\in\mathbb{N}}$
of independent Brownian motions (they are Brownian motions with respect to
$\left(  \mathcal{F}_{t}\right)  $ but in general they do not generate
$\left(  \mathcal{F}_{t}\right)  $), and given a generic initial condition in
a suitable class, then two solutions adapted to $\left(  \mathcal{F}%
_{t}\right)  $ coincide. Here we prove only that two solutions coincide when
they are adapted to $\left(  \overline{\mathcal{G}}_{t}\right)  $.
\end{remark}

\section{Preliminaries on the operator $\mathcal{L}$ and Interpolation
inequalities\label{sect prelim}}

\subsection{The differential operator $\mathcal{L}$}

A key role is played by the differential operator $\mathcal{L}$ defined by
(\ref{def L}). We state here its property of uniform ellipticity, based on the
assumptions on $Q$.

\begin{proposition}
\label{Prop ellipticity}Assume $Q$ to be twice continuously differentiable,
bounded with bounded first and second derivatives, and%
\[
Q\left(  x,x\right)  \geq\nu Id_{\mathbb{R}^{d}}%
\]
for some $\nu>0$, uniformly in $x\in\mathbb{R}^{d}$. Then $\mathcal{L}$ is
well defined, uniformly elliptic. In particular, there exists $C>0$ such that
\[
-\int_{\mathbb{R}^{d}}\mathcal{L}B\left(  x\right)  \cdot B\left(  x\right)
dx\geq\frac{\nu}{2}\int_{\mathbb{R}^{d}}\left\vert DB\left(  x\right)
\right\vert ^{2}dx-C\int_{\mathbb{R}^{d}}\left\vert B\left(  x\right)
\right\vert ^{2}dx
\]
for all $B\in W^{1,2}\left(  \mathbb{R}^{d},\mathbb{R}^{d}\right)  $. Moreover
it has the form%
\begin{align*}
\left(  \mathcal{L}B\right)  ^{\alpha}\left(  x\right)   &  =\sum_{i,j=1}%
^{d}a_{ij}\left(  x\right)  \partial_{i}\partial_{j}B^{\alpha}\left(
x\right)  \\
&  +\sum_{i,\beta=1}^{d}b_{i}^{\alpha\beta}\left(  x\right)  \partial
_{i}B^{\beta}\left(  x\right)  +\sum_{\beta=1}^{d}c^{\alpha\beta}\left(
x\right)  B^{\beta}\left(  x\right)
\end{align*}
where $a_{ij}$ is twice continuously differentiable, bounded with bounded
first and second derivatives, $b_{i}^{\alpha\beta}$ is continuously
differentiable, bounded with bounded first derivatives and $c^{\alpha\beta}$
is bounded continuous. The formal adjoint operator $\mathcal{L}^{\ast}$ given
by
\begin{align*}
\mathcal{L}^{\ast}\phi\left(  x\right)    & =\sum_{i,j=1}^{d}\partial
_{i}\partial_{j}\left(  a_{ij}\left(  x\right)  \phi^{\alpha}\left(  x\right)
\right)  \\
& -\sum_{i,\beta=1}^{d}\partial_{i}\left(  b_{i}^{\alpha\beta}\left(
x\right)  \phi^{\beta}\left(  x\right)  \right)  +\sum_{\beta=1}^{d}%
c^{\alpha\beta}\left(  x\right)  \phi^{\beta}\left(  x\right)
\end{align*}
maps vector fields $\phi$ that are twice continuously differentiable, bounded
with bounded first and second derivatives, into vector fields $\mathcal{L}%
^{\ast}\phi$ hat are bounded continuous.
\end{proposition}

We prepare the proof by the explicit computation of $\left[  \sigma
_{k},\left[  \sigma_{k},B\right]  \right]  $. We have%
\begin{align*}
\left[  \sigma_{k},\left[  \sigma_{k},B\right]  \right]   &  =\left(
\sigma_{k}\cdot\nabla\right)  \left[  \sigma_{k},B\right]  -\left(  \left[
\sigma_{k},B\right]  \cdot\nabla\right)  \sigma_{k}\\
&  =\left(  \sigma_{k}\cdot\nabla\right)  \left(  \sigma_{k}\cdot
\nabla\right)  B_{t}-\left(  \sigma_{k}\cdot\nabla\right)  \left(  B_{t}%
\cdot\nabla\right)  \sigma_{k}-\left(  \left(  \sigma_{k}\cdot\nabla\right)
B_{t}\cdot\nabla\right)  \sigma_{k}+\left(  \left(  B_{t}\cdot\nabla\right)
\sigma_{k}\cdot\nabla\right)  \sigma_{k}.
\end{align*}
All terms can be expressed by means of $Q$, after the following remarks. The
function $Q\left(  x,y\right)  $ is defined on $\mathbb{R}^{d}\times
\mathbb{R}^{d}$ with values in matrices $\mathbb{R}^{d\times d}$. When we
differentiate $Q^{\alpha\beta}\left(  x,y\right)  $ with respect to the first
set of components, we write $\left(  \partial_{i}^{\left(  1\right)
}Q^{\alpha\beta}\right)  \left(  x,y\right)  $:%
\[
\left(  \partial_{i}^{\left(  1\right)  }Q^{\alpha\beta}\right)  \left(
x,y\right)  =\lim_{\epsilon\rightarrow0}\frac{Q^{\alpha\beta}\left(
x+\epsilon e_{i},y\right)  -Q^{\alpha\beta}\left(  x,y\right)  }{\epsilon}%
\]
while when we differentiate $Q^{\alpha\beta}\left(  x,y\right)  $ with respect
to the second set of components, we write $\left(  \partial_{i}^{\left(
2\right)  }Q^{\alpha\beta}\right)  \left(  x,y\right)  $. We have%
\begin{align*}
\left(  \partial_{i}^{\left(  1\right)  }Q^{\alpha\beta}\right)  \left(
x,y\right)   &  =\partial_{x_{i}}\left(  Q^{\alpha\beta}\left(  x,y\right)
\right)  =\sum_{k=1}^{\infty}\left(  \partial_{i}\sigma_{k}^{\alpha}\right)
\left(  x\right)  \sigma_{k}^{\beta}\left(  y\right) \\
\left(  \partial_{i}^{\left(  2\right)  }Q^{\alpha\beta}\right)  \left(
x,y\right)   &  =\partial_{y_{i}}\left(  Q^{\alpha\beta}\left(  x,y\right)
\right)  =\sum_{k=1}^{\infty}\sigma_{k}^{\alpha}\left(  x\right)  \left(
\partial_{i}\sigma_{k}^{\beta}\right)  \left(  y\right)  .
\end{align*}
Hence, when we evaluate at $y=x$,%
\begin{align*}
\sum_{k=1}^{\infty}\left(  \partial_{i}\sigma_{k}^{\alpha}\right)  \left(
x\right)  \sigma_{k}^{\beta}\left(  x\right)   &  =\left(  \partial
_{i}^{\left(  1\right)  }Q^{\alpha\beta}\right)  \left(  x,x\right) \\
\sum_{k=1}^{\infty}\sigma_{k}^{\alpha}\left(  x\right)  \left(  \partial
_{i}\sigma_{k}^{\beta}\right)  \left(  x\right)   &  =\left(  \partial
_{i}^{\left(  2\right)  }Q^{\alpha\beta}\right)  \left(  x,x\right)  .
\end{align*}
Similarly,%
\[
\left(  \partial_{j}^{\left(  1\right)  }\partial_{i}^{\left(  1\right)
}Q^{\alpha\beta}\right)  \left(  x,y\right)  =\partial_{x_{j}}\partial_{x_{i}%
}\left(  Q^{\alpha\beta}\left(  x,y\right)  \right)  =\sum_{k=1}^{\infty
}\left(  \partial_{j}\partial_{i}\sigma_{k}^{\alpha}\right)  \left(  x\right)
\sigma_{k}^{\beta}\left(  y\right)
\]
whence, at $y=x$,%
\[
\sum_{k=1}^{\infty}\left(  \partial_{j}\partial_{i}\sigma_{k}^{\alpha}\right)
\left(  x\right)  \sigma_{k}^{\beta}\left(  x\right)  =\left(  \partial
_{j}^{\left(  1\right)  }\partial_{i}^{\left(  1\right)  }Q^{\alpha\beta
}\right)  \left(  x,x\right)
\]
and so on for the other second derivatives. Let us denote by $\left[
\sigma_{k},\left[  \sigma_{k},B\right]  \right]  ^{\left(  \alpha\right)  }$
the $\alpha$-component of the vector $\left[  \sigma_{k},\left[  \sigma
_{k},B\right]  \right]  $.

\begin{lemma}%
\begin{align*}
&  \sum_{k}\left[  \sigma_{k},\left[  \sigma_{k},B\right]  \right]  ^{\left(
\alpha\right)  }\left(  x\right) \\
&  =\sum_{i,j=1}^{d}Q^{ij}\left(  x,x\right)  \partial_{i}\partial
_{j}B^{\alpha}\left(  x\right) \\
&  +\sum_{i=1}^{d}\sum_{\gamma=1}^{d}\partial_{\gamma}^{\left(  2\right)
}Q^{\gamma i}\left(  x,x\right)  \partial_{i}B^{\alpha}\left(  x\right)
-\sum_{i,\beta=1}^{d}2\left(  \partial_{\beta}^{\left(  2\right)  }Q^{i\alpha
}\right)  \left(  x,x\right)  \partial_{i}B^{\beta}\left(  x\right) \\
&  +\sum_{\beta,\gamma=1}^{d}\partial_{\beta}^{\left(  1\right)  }%
\partial_{\gamma}^{\left(  2\right)  }Q^{\gamma\alpha}\left(  x,x\right)
B^{\beta}\left(  x\right)  -\sum_{\gamma,\beta=1}^{d}\left(  \partial_{\gamma
}^{\left(  2\right)  }\partial_{\beta}^{\left(  2\right)  }Q^{\gamma\alpha
}\right)  \left(  x,x\right)  B^{\beta}\left(  x\right)  .
\end{align*}
Therefore the operator $\mathcal{L}$ has coefficients given by%
\begin{align*}
a_{ij}\left(  x\right)   &  =\frac{1}{2}Q^{ij}\left(  x,x\right) \\
b_{i}^{\alpha\beta}\left(  x\right)   &  =\frac{1}{2}\sum_{\gamma=1}%
^{d}\partial_{\gamma}^{\left(  2\right)  }Q^{\gamma i}\left(  x,x\right)
\delta_{\alpha\beta}-\left(  \partial_{\beta}^{\left(  2\right)  }Q^{i\alpha
}\right)  \left(  x,x\right) \\
c^{\alpha\beta}\left(  x\right)   &  =\frac{1}{2}\sum_{\gamma=1}^{d}%
\partial_{\beta}^{\left(  1\right)  }\partial_{\gamma}^{\left(  2\right)
}Q^{\gamma\alpha}\left(  x,x\right)  -\frac{1}{2}\sum_{\gamma=1}^{d}\left(
\partial_{\gamma}^{\left(  2\right)  }\partial_{\beta}^{\left(  2\right)
}Q^{\gamma\alpha}\right)  \left(  x,x\right)  .
\end{align*}

\end{lemma}

\begin{proof}%
\begin{align*}
&  \sum_{k}\left(  \sigma_{k}\cdot\nabla\right)  \left(  \sigma_{k}\cdot
\nabla\right)  B_{t}^{i}-\sum_{k}\left(  \sigma_{k}\cdot\nabla\right)  \left(
B_{t}\cdot\nabla\right)  \sigma_{k}^{i}-\sum_{k}\left(  \left(  \sigma
_{k}\cdot\nabla\right)  B_{t}\cdot\nabla\right)  \sigma_{k}^{i}+\sum
_{k}\left(  \left(  B_{t}\cdot\nabla\right)  \sigma_{k}\cdot\nabla\right)
\sigma_{k}^{i}\\
&  =\sum_{k}\sum_{\alpha\beta}\left(  \sigma_{k}^{\alpha}\partial_{\alpha
}\left(  \sigma_{k}^{\beta}\partial_{\beta}B_{t}^{i}\right)  -\sigma
_{k}^{\alpha}\partial_{\alpha}\left(  B_{t}^{\beta}\partial_{\beta}\sigma
_{k}^{i}\right)  -\sigma_{k}^{\alpha}\partial_{\alpha}B_{t}^{\beta}%
\partial_{\beta}\sigma_{k}^{i}+B_{t}^{\alpha}\partial_{\alpha}\sigma
_{k}^{\beta}\partial_{\beta}\sigma_{k}^{i}\right) \\
&  =\sum_{k}\sum_{\alpha\beta}\left(  \sigma_{k}^{\alpha}\sigma_{k}^{\beta
}\partial_{\alpha}\partial_{\beta}B_{t}^{i}+\sigma_{k}^{\alpha}\partial
_{\alpha}\sigma_{k}^{\beta}\partial_{\beta}B_{t}^{i}-\sigma_{k}^{\alpha}%
B_{t}^{\beta}\partial_{\alpha}\partial_{\beta}\sigma_{k}^{i}\right) \\
&  +\sum_{k}\sum_{\alpha\beta}\left(  -\sigma_{k}^{\alpha}\partial_{\alpha
}B_{t}^{\beta}\partial_{\beta}\sigma_{k}^{i}-\sigma_{k}^{\alpha}%
\partial_{\alpha}B_{t}^{\beta}\partial_{\beta}\sigma_{k}^{i}+B_{t}^{\alpha
}\partial_{\alpha}\sigma_{k}^{\beta}\partial_{\beta}\sigma_{k}^{i}\right) \\
&  =\sum_{\alpha\beta}Q^{\alpha\beta}\left(  x,x\right)  \partial_{\alpha
}\partial_{\beta}B_{t}^{i}+\left(  \partial_{\alpha}^{\left(  2\right)
}Q^{\alpha\beta}\right)  \left(  x,x\right)  \partial_{\beta}B_{t}^{i}-\left(
\partial_{\alpha}^{\left(  2\right)  }\partial_{\beta}^{\left(  2\right)
}Q^{\alpha i}\right)  \left(  x,x\right)  B_{t}^{\beta}\\
&  +\sum_{\alpha\beta}\left(  -2\partial_{\beta}^{\left(  2\right)  }Q^{\alpha
i}\partial_{\alpha}B_{t}^{\beta}+\partial_{\alpha}^{\left(  1\right)
}\partial_{\beta}^{\left(  2\right)  }Q^{\beta i}B_{t}^{\alpha}\right)  .
\end{align*}
The result of the lemma is just a rewriting of this expression.
\end{proof}

Now, we do the proof of the Proposition \ref{Prop ellipticity}.

\begin{proof}
[Proof of Proposition \ref{Prop ellipticity}]Let us set%
\[
R_{0}:=\sum_{\alpha\beta}\left(  \left(  \partial_{\alpha}^{\left(  2\right)
}Q^{\alpha\beta}\right)  \left(  x,x\right)  \partial_{\beta}B_{t}^{i}-\left(
\partial_{\alpha}^{\left(  2\right)  }\partial_{\beta}^{\left(  2\right)
}Q^{\alpha i}\right)  \left(  x,x\right)  B_{t}^{\beta}\right)  B
\]%
\[
+\sum_{\alpha\beta}\left(  \left(  -2\partial_{\beta}^{\left(  2\right)
}Q^{\alpha i}\partial_{\alpha}B_{t}^{\beta}+\partial_{\alpha}^{\left(
1\right)  }\partial_{\beta}^{\left(  2\right)  }Q^{\beta i}B_{t}^{\alpha
}\right)  \right)  B
\]
Then we have
\begin{align*}
-\int_{\mathbb{R}^{d}}\mathcal{L}B\left(  x\right)  \cdot B\left(  x\right)
dx &  =-\sum_{i}\int_{\mathbb{R}^{d}}\sum_{\alpha\beta}Q^{\alpha\beta}\left(
x,x\right)  \partial_{\alpha}\partial_{\beta}B^{i}\left(  x\right)
B^{i}\left(  x\right)  dx+R_{0}\\
&  =\sum_{i}\int_{\mathbb{R}^{d}}\sum_{\alpha\beta}Q^{\alpha\beta}\left(
x,x\right)  \partial_{\beta}B^{i}\left(  x\right)  \partial_{\alpha}%
B^{i}\left(  x\right)  dx\\
&  +\sum_{i}\int_{\mathbb{R}^{d}}\sum_{\alpha\beta}\partial_{\alpha}%
Q^{\alpha\beta}\left(  x,x\right)  \partial_{\beta}B^{i}\left(  x\right)
B^{i}\left(  x\right)  dx+R_{0}\\
&  \geq\nu\sum_{i}\int_{\mathbb{R}^{d}}\left\vert \nabla B^{i}\left(
x\right)  \right\vert ^{2}dx-\sum_{i\alpha\beta}\int_{\mathbb{R}^{d}%
}\left\vert \partial_{\alpha}Q^{\alpha\beta}\left(  x,x\right)  \right\vert
\left\vert \partial_{\beta}B^{i}\left(  x\right)  \right\vert \left\vert
B^{i}\left(  x\right)  \right\vert dx-\left\vert R_{0}\right\vert \\
&  =\nu\int_{\mathbb{R}^{d}}\left\vert DB\left(  x\right)  \right\vert
^{2}dx-R_{1}-\left\vert R_{0}\right\vert
\end{align*}
with $R_{1}$ defined by the identity. The estimates on $\left\vert
R_{0}\right\vert $ are similar to the estimate on $R_{1}$, so we limit
ourselves to this one. We have%
\[
R_{1}\leq C_{1}\sum_{i\alpha\beta}\int_{\mathbb{R}^{d}}\left\vert
\partial_{\beta}B^{i}\left(  x\right)  \right\vert \left\vert B^{i}\left(
x\right)  \right\vert dx
\]
because we have assumed that $Q$ has bounded derivatives,
\[
\leq C_{2}\int_{\mathbb{R}^{d}}\left\vert DB\left(  x\right)  \right\vert
\left\vert B\left(  x\right)  \right\vert dx\leq\frac{\nu}{4}\int%
_{\mathbb{R}^{d}}\left\vert DB\left(  x\right)  \right\vert ^{2}dx+C_{3}%
\int_{\mathbb{R}^{d}}\left\vert B\left(  x\right)  \right\vert ^{2}dx.
\]
Here we have denoted by $C_{i}>0$ some constants, possibly depending on $\nu$
and other factors, but not on $B$. In the analogous estimates for $\left\vert
R_{0}\right\vert $,
\[
\left\vert R_{0}\right\vert \leq\frac{\nu}{4}\int_{\mathbb{R}^{d}}\left\vert
DB\left(  x\right)  \right\vert ^{2}dx+C_{4}\int_{\mathbb{R}^{d}}\left\vert
B\left(  x\right)  \right\vert ^{2}dx
\]
we use the assumption that the second derivatives of $Q$ are bounded. We
conclude that
\[
-\int_{\mathbb{R}^{d}}\mathcal{L}B\left(  x\right)  \cdot B\left(  x\right)
dx\geq\frac{\nu}{2}\int_{\mathbb{R}^{d}}\left\vert DB\left(  x\right)
\right\vert ^{2}dx-\left(  C_{3}+C_{4}\right)  \int_{\mathbb{R}^{d}}\left\vert
B\left(  x\right)  \right\vert ^{2}dx.
\]

\end{proof}

\subsection{Interpolation inequalities}

If we assume $v$ globally bounded, some estimates needed in the next sections
simplify. Since we want to treat $v$ of a certain integrability class, we need
some additional computations based on interpolation inequalities. We find
convenient to isolate here the precise inequalities used in the sequel.

\begin{lemma}
\label{lemma interp}If $f,h\in W^{1,2}\left(  \mathbb{R}^{d}\right)  $ and
$g\in L^{p}\left(  \mathbb{R}^{d}\right)  $ for some $p>d$, then%
\[
\int_{\mathbb{R}^{d}}f\left(  x\right)  g\left(  x\right)  \partial
_{i}h\left(  x\right)  dx\leq C\left\Vert g\right\Vert _{L^{p}\left(
\mathbb{R}^{d}\right)  }\left\Vert f\right\Vert _{W^{1,2}\left(
\mathbb{R}^{d}\right)  }\left\Vert h\right\Vert _{W^{1,2}\left(
\mathbb{R}^{d}\right)  }%
\]
where $C>0$ is a constant independent of $f,g,h$ and for every $\epsilon>0$
there is a constant $C_{\epsilon}>0$ such that
\[
\int_{\mathbb{R}^{d}}f\left(  x\right)  g\left(  x\right)  \partial
_{i}h\left(  x\right)  dx\leq\epsilon\left\Vert h\right\Vert _{W^{1,2}\left(
\mathbb{R}^{d}\right)  }^{2}+\epsilon\left\Vert f\right\Vert _{W^{1,2}\left(
\mathbb{R}^{d}\right)  }^{2}+C_{\epsilon}\left\Vert g\right\Vert
_{L^{p}\left(  \mathbb{R}^{d}\right)  }^{\frac{2p}{p-d}}\left\Vert
f\right\Vert _{L^{2}\left(  \mathbb{R}^{d}\right)  }^{2}.
\]

\end{lemma}

\begin{proof}
\textbf{Step 1}.
\[
\int_{\mathbb{R}^{d}}\left\vert f\right\vert ^{2}\left\vert g\right\vert
^{2}dx\leq\left\Vert f\right\Vert _{W^{1,2}\left(  \mathbb{R}^{d}\right)
}^{\frac{6}{p}}\left(  C\left\Vert g\right\Vert _{L^{p}\left(  \mathbb{R}%
^{d}\right)  }^{2}\left\Vert f\right\Vert _{L^{2}\left(  \mathbb{R}%
^{d}\right)  }^{2-\frac{6}{p}}\right)  .
\]
Indeed,%
\[
\int_{\mathbb{R}^{d}}\left\vert f\right\vert ^{2}\left\vert g\right\vert
^{2}dx\leq\left(  \int_{\mathbb{R}^{d}}\left\vert g\right\vert ^{p}dx\right)
^{\frac{2}{p}}\left(  \int_{\mathbb{R}^{d}}\left\vert f\right\vert ^{\frac
{2p}{p-2}}dx\right)  ^{\frac{p-2}{p}}=\left\Vert g\right\Vert _{L^{p}\left(
\mathbb{R}^{d}\right)  }^{2}\left\Vert f\right\Vert _{L^{\frac{2p}{p-2}%
}\left(  \mathbb{R}^{d}\right)  }^{2}.
\]
Recall now that, by Sobolev embedding theorem,%
\[
\left\Vert f\right\Vert _{L^{\frac{2p}{p-2}}\left(  \mathbb{R}^{d}\right)
}\leq C\left\Vert f\right\Vert _{W^{\frac{d}{p},2}\left(  \mathbb{R}%
^{d}\right)  }%
\]
(because, in general, $W^{s,2}\subset L^{r}$ for $\frac{1}{r}=\frac{1}%
{2}-\frac{s}{d}$, hence in our case $s=\frac{d}{2}-\frac{d\left(  p-2\right)
}{2p}=\frac{dp-d\left(  p-2\right)  }{2p}=\frac{d}{p}$) and the interpolation
inequality gives, for $\alpha\in\left(  0,1\right)  $,
\[
\left\Vert f\right\Vert _{W^{\alpha,2}\left(  \mathbb{R}^{d}\right)  }\leq
C\left\Vert f\right\Vert _{L^{2}\left(  \mathbb{R}^{d}\right)  }^{1-\alpha
}\left\Vert f\right\Vert _{W^{1,2}\left(  \mathbb{R}^{d}\right)  }^{\alpha}%
\]
hence%
\[
\left\Vert f\right\Vert _{W^{\frac{d}{p},2}\left(  \mathbb{R}^{d}\right)
}\leq C\left\Vert f\right\Vert _{L^{2}\left(  \mathbb{R}^{d}\right)
}^{1-\frac{d}{p}}\left\Vert f\right\Vert _{W^{1,2}\left(  \mathbb{R}%
^{d}\right)  }^{\frac{d}{p}}.
\]
Summarizing,%
\[
\int_{\mathbb{R}^{d}}\left\vert f\right\vert ^{2}\left\vert g\right\vert
^{2}dx\leq\left\Vert f\right\Vert _{W^{1,2}\left(  \mathbb{R}^{d}\right)
}^{\frac{2d}{p}}\left(  C\left\Vert g\right\Vert _{L^{p}\left(  \mathbb{R}%
^{d}\right)  }^{2}\left\Vert f\right\Vert _{L^{2}\left(  \mathbb{R}%
^{d}\right)  }^{2-\frac{2d}{p}}\right)  .
\]

\textbf{Step 2}. Obviously%
\[
\int_{\mathbb{R}^{d}}f\left(  x\right)  g\left(  x\right)  \partial
_{i}h\left(  x\right)  dx\leq C\int_{\mathbb{R}^{d}}\left\vert f\left(
x\right)  \right\vert \left\vert g\left(  x\right)  \right\vert \left\vert
\partial_{i}h\left(  x\right)  \right\vert dx\leq C\left\Vert h\right\Vert
_{W^{1,2}\left(  \mathbb{R}^{d}\right)  }\left(  \int_{\mathbb{R}^{d}%
}\left\vert f\right\vert ^{2}\left\vert g\right\vert ^{2}dx\right)  ^{1/2}.
\]
Hence from the inequality of Step 1,%
\[
\int_{\mathbb{R}^{d}}f\left(  x\right)  g\left(  x\right)  \partial
_{i}h\left(  x\right)  dx\leq C\left\Vert h\right\Vert _{W^{1,2}\left(
\mathbb{R}^{d}\right)  }\left\Vert f\right\Vert _{W^{1,2}\left(
\mathbb{R}^{d}\right)  }\left\Vert g\right\Vert _{L^{p}\left(  \mathbb{R}%
^{d}\right)  }.
\]

\textbf{Step 3}. Again%
\begin{align*}
\int_{\mathbb{R}^{d}}f\left(  x\right)  g\left(  x\right)  \partial
_{i}h\left(  x\right)  dx  &  \leq\epsilon\left\Vert h\right\Vert
_{W^{1,2}\left(  \mathbb{R}^{d}\right)  }^{2}+C_{\epsilon}\int_{\mathbb{R}%
^{d}}\left\vert f\right\vert ^{2}\left\vert g\right\vert ^{2}dx.\\
&  \leq\epsilon\left\Vert h\right\Vert _{W^{1,2}\left(  \mathbb{R}^{d}\right)
}^{2}+C_{\epsilon}\left\Vert f\right\Vert _{W^{1,2}\left(  \mathbb{R}%
^{d}\right)  }^{\frac{2d}{p}}\left(  C\left\Vert g\right\Vert _{L^{p}\left(
\mathbb{R}^{d}\right)  }^{2}\left\Vert f\right\Vert _{L^{2}\left(
\mathbb{R}^{d}\right)  }^{2-\frac{2d}{p}}\right)  .
\end{align*}
By Young inequality $ab\leq\delta a^{r}+C_{\delta}b^{r^{\prime}}$, $\frac
{1}{r}+\frac{1}{r^{\prime}}=1$, we have%
\begin{align*}
&  \left\Vert f\right\Vert _{W^{1,2}\left(  \mathbb{R}^{d}\right)  }%
^{\frac{2d}{p}}\left(  \left\Vert g\right\Vert _{L^{p}\left(  \mathbb{R}%
^{d}\right)  }^{2}\left\Vert f\right\Vert _{L^{2}\left(  \mathbb{R}%
^{d}\right)  }^{2-\frac{2d}{p}}\right) \\
&  \leq\delta\left\Vert f\right\Vert _{W^{1,2}\left(  \mathbb{R}^{d}\right)
}^{2}+C_{\delta}\left\Vert g\right\Vert _{L^{p}\left(  \mathbb{R}^{d}\right)
}^{\frac{2p}{p-d}}\left\Vert f\right\Vert _{L^{2}\left(  \mathbb{R}%
^{d}\right)  }^{2}%
\end{align*}
where was used with $r=\frac{p}{d}$, hence $r^{\prime}=\frac{p}{p-d}$. By a
proper choice of $\epsilon$ and $\delta$, we get the last inequality of the lemma.
\end{proof}

\section{Existence\label{section existence}}

We divide the proof into three steps. First, taking a regular approximation
the vector field $v$ and the initial condition $B_{0}$, we prove that the
problem (\ref{Adve}) admits a regular solution $B^{\varepsilon}$. Then, in the
second step, we prove some energy estimates on the regularized problem
independently of $\epsilon$. Finally, by weak compactness, we extract a
subsequence which converges weakly to a solution; weak convergence is
sufficient since the equation is linear, and $v$ has been regularized in the
strong topology. The estimates on the regularized problem are sufficient for
the convergence in the weak formulation of both the stochastic advection
equation and the parabolic equation for the expected values.

Assume $v$ and $B_{0}$\ smooth and let $B$ be a smooth solution (it exists by
stochastic flows, see the next remarks). In this section we prove a priori
estimates which depend only on the norms $\left\Vert v\right\Vert _{L^{\infty
}\left(  0,T;L^{p}\left(  \mathbb{R}^{d}\right)  \right)  }^{2}$ and
$\left\Vert B_{0}\right\Vert _{L^{4}}$. By a classical procedure we shall
deduce the existence of a quasi-regular weak solution: given our non-smooth
data $v$ and $B_{0}$, taken a family of standard symmetric mollifiers
$\{\rho_{\varepsilon}\}_{\varepsilon}$, we define the family of regularized
coefficients $v^{\epsilon}(t,x)=[v(t,\cdot)\ast\rho_{\varepsilon}(\cdot)](x)$
and initial conditions $B_{0}^{\varepsilon}(x)=[B_{0}(\cdot)\ast
\rho_{\varepsilon}(\cdot)](x)$; we prove the estimates for the corresponding
solutions $B^{\varepsilon}(t,x)$ and extract a subsequence which converges
weakly;\ then we pass to the limit just by weak convergence since the equation
is linear and the coefficients $v^{\epsilon}(t,x)$ converge strongly.

Let us remark about the existence of a smooth solution $B^{\varepsilon}(t,x)$.
By the classical results of \cite{Ku},\ the equation
\[
dX_{t}^{\varepsilon}=v^{\varepsilon}(t,X_{t}^{\varepsilon})\,dt+\sum_{k}%
\sigma_{k}\left(  X_{t}^{\varepsilon}\right)  \circ dW_{t}^{k}\,,\hspace
{1cm}X_{0}=x
\]
generates a stochastic flow of smooth diffeomorphisms $\Phi_{t}^{\varepsilon
}\left(  x\right)  $, with inverse $\Psi_{t}^{\varepsilon}\left(  x\right)  $.
Then formula
\begin{equation}
B^{\varepsilon}(t,x)=(D\Phi_{t}^{\varepsilon})(\Psi_{t}^{\varepsilon}\left(
x\right)  )B_{0}^{\varepsilon}\left(  \Psi_{t}^{\varepsilon}\left(  x\right)
\right)  \label{repr formula}%
\end{equation}
gives us a smooth solution.

\subsection{Estimates on $B^{\varepsilon}$ in $L^{2}\left(  \left[
0,T\right]  \times\Omega,L_{loc}^{2}(\mathbb{R}^{d})\right)  $}

In the first part of this section we denote $B^{\varepsilon}$ and
$v^{\varepsilon}$ by $B$ and $v$, to simplify notations. The constants in the
estimates are independent of $\varepsilon$. The strategy used in this section
is inspired by \cite{Beck}.

We formally work on the Stratonovich formulation but the same computations can
be done in a rigorous, although blind, way on the It\^{o} form. Thus we start,
componentwise, from
\[
dB^{\alpha}+\left[  v,B\right]  ^{\alpha}\ dt+\sum_{k=1}^{\infty}\left[
\sigma_{k},B\right]  ^{\alpha}\circ dW_{t}^{k}=0.
\]
We have, for any $\alpha,\beta=1,...,d$,%
\[
d\left(  B^{\alpha}B^{\beta}\right)  +\left[  v,B\right]  ^{\alpha}B^{\beta
}\ dt+\sum_{k=1}^{\infty}\left[  \sigma_{k},B\right]  ^{\alpha}B^{\beta}\circ
dW_{t}^{k}%
\]%
\[
+B^{\alpha}\left[  v,B\right]  ^{\beta}\ dt+\sum_{k=1}^{\infty}B^{\alpha
}\left[  \sigma_{k},B\right]  ^{\beta}\circ dW_{t}^{k}=0.
\]
Notice that%
\begin{align*}
&  \left[  A,B\right]  ^{\alpha}B^{\beta}+\left[  A,B\right]  ^{\beta
}B^{\alpha}\\
&  =\left(  A\cdot\nabla B^{\alpha}\right)  B^{\beta}-\left(  B\cdot\nabla
A^{\alpha}\right)  B^{\beta}+\left(  A\cdot\nabla B^{\beta}\right)  B^{\alpha
}-\left(  B\cdot\nabla A^{\beta}\right)  B^{\alpha}\\
&  =\left(  A\cdot\nabla\right)  \left(  B^{\alpha}B^{\beta}\right)  -\left(
B\cdot\nabla A^{\alpha}\right)  B^{\beta}-\left(  B\cdot\nabla A^{\beta
}\right)  B^{\alpha}\\
&  =\left(  A\cdot\nabla\right)  \left(  B^{\alpha}B^{\beta}\right)  -\left(
B^{\beta}B\cdot\nabla\right)  A^{\alpha}-\left(  B^{\alpha}B\cdot
\nabla\right)  A^{\beta}%
\end{align*}
hence%
\[
d\left(  B^{\alpha}B^{\beta}\right)  +\left(  v\cdot\nabla\right)  \left(
B^{\alpha}B^{\beta}\right)  \ dt+\sum_{k=1}^{\infty}\left(  \sigma_{k}%
\cdot\nabla\right)  \left(  B^{\alpha}B^{\beta}\right)  \circ dW_{t}^{k}%
\]%
\[
-\left(  \left(  B^{\beta}B\cdot\nabla\right)  v^{\alpha}+\left(  B^{\alpha
}B\cdot\nabla\right)  v^{\beta}\right)  \ dt-\sum_{k=1}^{\infty}\left(
\left(  B^{\beta}B\cdot\nabla\right)  \sigma_{k}^{\alpha}+\left(  B^{\alpha
}B\cdot\nabla\right)  \sigma_{k}^{\beta}\right)  \circ dW_{t}^{k}=0.
\]
We go now to It\^{o} form; the corrections are, on the LHS:%
\begin{align*}
\frac{1}{2}\sum_{k=1}^{\infty}d\left\langle \left(  \sigma_{k}\cdot
\nabla\right)  \left(  B^{\alpha}B^{\beta}\right)  ,W^{k}\right\rangle _{t} &
=\frac{1}{2}\sum_{k=1}^{\infty}\left(  \sigma_{k}\cdot\nabla\right)
d\left\langle B^{\alpha}B^{\beta},W^{k}\right\rangle _{t}\\
&  =-\frac{1}{2}\sum_{k=1}^{\infty}\left(  \sigma_{k}\cdot\nabla\right)
\left(  \sigma_{k}\cdot\nabla\right)  \left(  B^{\alpha}B^{\beta}\right)  dt\\
&  +\frac{1}{2}\sum_{k=1}^{\infty}\left(  \sigma_{k}\cdot\nabla\right)
\left(  \left(  B^{\beta}B\cdot\nabla\right)  \sigma_{k}^{\alpha}+\left(
B^{\alpha}B\cdot\nabla\right)  \sigma_{k}^{\beta}\right)  dt
\end{align*}
and%
\begin{align*}
&  -\frac{1}{2}\sum_{k=1}^{\infty}d\left\langle \left(  B^{\beta}B\cdot
\nabla\right)  \sigma_{k}^{\alpha}+\left(  B^{\alpha}B\cdot\nabla\right)
\sigma_{k}^{\beta},W^{k}\right\rangle _{t}\\
&  =-\frac{1}{2}\sum_{k=1}^{\infty}\sum_{i=1}^{d}d\left\langle B^{\beta}%
B^{i},W^{k}\right\rangle _{t}\partial_{i}\sigma_{k}^{\alpha}-\frac{1}{2}%
\sum_{k=1}^{\infty}\sum_{i=1}^{d}d\left\langle B^{\alpha}B^{i},W^{k}%
\right\rangle _{t}\partial_{i}\sigma_{k}^{\beta}\\
&  =\frac{1}{2}\sum_{k=1}^{\infty}\sum_{i=1}^{d}\left(  \sigma_{k}\cdot
\nabla\right)  \left(  B^{\beta}B^{i}\right)  \partial_{i}\sigma_{k}^{\alpha
}dt-\frac{1}{2}\sum_{k=1}^{\infty}\sum_{i=1}^{d}\left(  \left(  B^{\beta
}B\cdot\nabla\right)  \sigma_{k}^{i}+\left(  B^{i}B\cdot\nabla\right)
\sigma_{k}^{\beta}\right)  \partial_{i}\sigma_{k}^{\alpha}dt\\
&  +\frac{1}{2}\sum_{k=1}^{\infty}\sum_{i=1}^{d}\left(  \sigma_{k}\cdot
\nabla\right)  \left(  B^{\alpha}B^{i}\right)  \partial_{i}\sigma_{k}^{\beta
}dt-\frac{1}{2}\sum_{k=1}^{\infty}\sum_{i=1}^{d}\left(  \left(  B^{i}%
B\cdot\nabla\right)  \sigma_{k}^{\alpha}+\left(  B^{\alpha}B\cdot
\nabla\right)  \sigma_{k}^{i}\right)  \partial_{i}\sigma_{k}^{\beta}dt.
\end{align*}
Let us properly summarize the large amount of terms of the
It\^{o}-Stratonovich corrections. We have the identity%
\begin{align*}
& -\frac{1}{2}\sum_{k=1}^{\infty}\left(  \sigma_{k}\cdot\nabla\right)  \left(
\sigma_{k}\cdot\nabla\right)  \left(  B^{\alpha}B^{\beta}\right)  \\
& =-\frac{1}{2}\sum_{k=1}^{\infty}\sum_{i,j=1}^{d}\sigma_{k}^{i}\sigma_{k}%
^{j}\partial_{i}\partial_{j}\left(  B^{\alpha}B^{\beta}\right)  -\frac{1}%
{2}\sum_{k=1}^{\infty}\sum_{i=1}^{d}\sum_{j=1}^{d}\sigma_{k}^{i}\partial
_{i}\sigma_{k}^{j}\partial_{j}\left(  B^{\alpha}B^{\beta}\right)  \\
& =-\frac{1}{2}\sum_{i,j=1}^{d}Q^{ij}\left(  x,x\right)  \partial_{i}%
\partial_{j}\left(  B^{\alpha}B^{\beta}\right)  +\theta\cdot\nabla\left(
B^{\alpha}B^{\beta}\right)
\end{align*}
where
\[
\theta^{j}=-\frac{1}{2}\sum_{k=1}^{\infty}\sigma_{k}\cdot\nabla\sigma_{k}^{j}.
\]
Then, with the additional notations%
\begin{align*}
\theta_{i}^{\alpha}  & =-\sum_{k=1}^{\infty}\sigma_{k}\partial_{i}\sigma
_{k}^{\alpha}\\
\eta_{i}^{\alpha}  & =\frac{1}{2}\sum_{k=1}^{\infty}\left(  \sigma_{k}%
\cdot\nabla\right)  \partial_{i}\sigma_{k}^{\alpha}+\frac{1}{2}\sum
_{k=1}^{\infty}\sum_{j=1}^{d}\partial_{i}\sigma_{k}^{j}\partial_{j}\sigma
_{k}^{\alpha}\\
\eta_{i,j}^{\alpha,\beta}  & =\sum_{k=1}^{\infty}\partial_{j}\sigma_{k}%
^{\beta}\partial_{i}\sigma_{k}^{\alpha}%
\end{align*}
the sum of all correction terms can be written in the form%
\begin{align*}
&  -\frac{1}{2}\sum_{i,j=1}^{d}Q^{ij}\left(  x,x\right)  \partial_{i}%
\partial_{j}\left(  B^{\alpha}B^{\beta}\right)  dt+\theta\cdot\nabla\left(
B^{\alpha}B^{\beta}\right)  dt\\
&  -\sum_{i=1}^{d}\theta_{i}^{\alpha}\cdot\nabla\left(  B^{\beta}B^{i}\right)
dt-\sum_{i=1}^{d}\theta_{i}^{\beta}\nabla\left(  B^{\alpha}B^{i}\right)  dt\\
&  -\sum_{i=1}^{d}\eta_{i}^{\alpha}B^{\beta}B^{i}dt-\sum_{i=1}^{d}\eta
_{i}^{\beta}B^{\alpha}B^{i}dt-\sum_{i,j=1}^{d}\eta_{i,j}^{\alpha,\beta}%
B^{i}B^{j}dt
\end{align*}

Summarizing, and moving all the correction terms on the RHS, we get%
\begin{align*}
&  d\left(  B^{\alpha}B^{\beta}\right)  +\left(  v\cdot\nabla\right)  \left(
B^{\alpha}B^{\beta}\right)  \ dt+\sum_{k=1}^{\infty}\left(  \sigma_{k}%
\cdot\nabla\right)  \left(  B^{\alpha}B^{\beta}\right)  dW_{t}^{k}\\
&  -\left(  \left(  B^{\beta}B\cdot\nabla\right)  v^{\alpha}+\left(
B^{\alpha}B\cdot\nabla\right)  v^{\beta}\right)  \ dt-\sum_{k=1}^{\infty
}\left(  \left(  B^{\beta}B\cdot\nabla\right)  \sigma_{k}^{\alpha}+\left(
B^{\alpha}B\cdot\nabla\right)  \sigma_{k}^{\beta}\right)  dW_{t}^{k}%
\end{align*}%
\begin{align*}
&  =\frac{1}{2}\sum_{i,j=1}^{d}Q^{ij}\left(  x,x\right)  \partial_{i}%
\partial_{j}\left(  B^{\alpha}B^{\beta}\right)  dt+\theta\cdot\nabla\left(
B^{\alpha}B^{\beta}\right)  dt\\
&  -\sum_{i=1}^{d}\theta_{i}^{\alpha}\cdot\nabla\left(  B^{\beta}B^{i}\right)
dt-\sum_{i=1}^{d}\theta_{i}^{\beta}\nabla\left(  B^{\alpha}B^{i}\right)  dt\\
&  -\sum_{i=1}^{d}\eta_{i}^{\alpha}B^{\beta}B^{i}dt-\sum_{i=1}^{d}\eta
_{i}^{\beta}B^{\alpha}B^{i}dt-\sum_{i,j=1}^{d}\eta_{i,j}^{\alpha,\beta}%
B^{i}B^{j}dt.
\end{align*}
Then, taking expectation, we deduce%
\[
\frac{\partial}{\partial t}\mathbb{E}\left[  B^{\alpha}B^{\beta}\right]
+\left(  v\cdot\nabla\right)  \mathbb{E}\left[  B^{\alpha}B^{\beta}\right]
-\left(  \mathbb{E}\left[  B^{\beta}B\right]  \cdot\nabla\right)  v^{\alpha
}-\left(  \mathbb{E}\left[  B^{\alpha}B\right]  \cdot\nabla\right)  v^{\beta}%
\]%
\begin{align*}
&  =\frac{1}{2}\sum_{i,j=1}^{d}Q^{ij}\left(  x,x\right)  \partial_{i}%
\partial_{j}\mathbb{E}\left[  B^{\alpha}B^{\beta}\right]  +\theta\cdot
\nabla\mathbb{E}\left[  B^{\alpha}B^{\beta}\right]  \\
&  +\sum_{i=1}^{d}\theta_{i}^{\alpha}\cdot\nabla\mathbb{E}\left[  B^{\beta
}B^{i}\right]  +\sum_{i=1}^{d}\theta_{i}^{\beta}\nabla\mathbb{E}\left[
B^{\alpha}B^{i}\right]  \\
&  +\sum_{i=1}^{d}\eta_{j}^{\alpha}\mathbb{E}\left[  B^{\beta}B^{i}\right]
+\sum_{i=1}^{d}\eta_{j}^{\beta}\mathbb{E}\left[  B^{\alpha}B^{i}\right]
+\sum_{i,j=1}^{d}\eta_{i,j}^{\alpha,\beta}\mathbb{E}\left[  B^{i}B^{j}\right]
.
\end{align*}
Define $u_{\alpha\beta}\left(  t,x\right)  :=\mathbb{E}\left[  B^{\alpha
}\left(  t,x\right)  B^{\beta}\left(  t,x\right)  \right]  $. We have the
system of parabolic equations%
\begin{align*}
&  \frac{\partial u_{\alpha\beta}}{\partial t}+\left(  v\cdot\nabla\right)
u_{\alpha\beta}-\sum_{i=1}^{d}u_{\beta i}\partial_{i}v^{\alpha}-\sum_{i=1}%
^{d}u_{\alpha i}\partial_{i}v^{\beta}\\
&  =\frac{1}{2}\sum_{i,j=1}^{d}Q^{ij}\left(  x,x\right)  \partial_{i}%
\partial_{j}u_{\alpha\beta}+M^{\alpha\beta}\left(  u\right)
\end{align*}
where $u$ is the matrix-function $\left(  u_{ij}\right)  _{ij}$ and
$M^{\alpha\beta}\left(  u\right)  $ is a first order differential operator
with bounded continuous coefficients.

Then%
\begin{align*}
&  \frac{1}{2}\frac{d}{dt}\int u_{\alpha\beta}^{2}\left(  t,x\right)
dx-\frac{1}{2}\sum_{i,j=1}^{d}\int Q^{ij}\left(  x,x\right)  \partial
_{i}\partial_{j}u_{\alpha\beta}u_{\alpha\beta}dx\\
&  =\sum_{i=1}^{d}\int u_{\alpha\beta}u_{\beta i}\partial_{i}v^{\alpha}%
dx+\sum_{i=1}^{d}\int u_{\alpha\beta}u_{\alpha i}\partial_{i}v^{\beta}dx+\int
M^{\alpha\beta}\left(  u\right)  u_{\alpha\beta}dx.
\end{align*}
Similarly to Proposition \ref{Prop ellipticity} we have%
\[
-\frac{1}{2}\sum_{i,j=1}^{d}\int Q^{ij}\left(  x,x\right)  \partial
_{i}\partial_{j}u_{\alpha\beta}u_{\alpha\beta}dx\geq\frac{\nu}{2}%
\int\left\vert \nabla u_{\alpha\beta}\right\vert ^{2}dx-C\int\left\vert
u_{\alpha\beta}\right\vert ^{2}dx
\]
for some constants $\nu>0$, $C\geq0$, and similarly, for every $\epsilon>0$,%
\[
\int M^{\alpha\beta}\left(  u\right)  u_{\alpha\beta}dx\leq\epsilon
\sum_{\alpha,\beta=1}^{d}\int\left\vert \nabla u_{\alpha\beta}\right\vert
^{2}dx+C_{\epsilon}\sum_{\alpha,\beta=1}^{d}\int u_{\alpha\beta}^{2}dx
\]
hence%
\begin{align*}
&  \frac{1}{2}\frac{d}{dt}\int u_{\alpha\beta}^{2}\left(  t,x\right)
dx+\frac{\nu}{2}\int\left\vert \nabla u_{\alpha\beta}\right\vert
^{2}dx-\epsilon\sum_{\alpha,\beta=1}^{d}\int\left\vert \nabla u_{\alpha\beta
}\right\vert ^{2}dx\\
&  \leq C_{\epsilon}\sum_{\alpha,\beta=1}^{d}\int u_{\alpha\beta}^{2}%
dx+\sum_{i=1}^{d}\int u_{\alpha\beta}u_{\beta i}\partial_{i}v^{\alpha}%
dx+\sum_{i=1}^{d}\int u_{\alpha\beta}u_{\alpha i}\partial_{i}v^{\beta}dx.
\end{align*}
We deduce, with a proper choice of $\epsilon>0$,
\begin{align*}
&  \frac{1}{2}\frac{d}{dt}\sum_{\alpha,\beta=1}^{d}\int u_{\alpha\beta}%
^{2}dx+\frac{\nu}{4}\sum_{\alpha,\beta=1}^{d}\int\left\vert \nabla
u_{\alpha\beta}\right\vert ^{2}dx\\
&  \leq C\sum_{\alpha,\beta=1}^{d}\int u_{\alpha\beta}^{2}dx+\sum
_{i,\alpha,\beta=1}^{d}\int u_{\alpha\beta}u_{\beta i}\partial_{i}v^{\alpha
}dx+\sum_{i=1}^{d}\int u_{\alpha\beta}u_{\alpha i}\partial_{i}v^{\beta}dx.
\end{align*}
Let us treat only the term $\int u_{\alpha\beta}u_{\beta i}\partial
_{i}v^{\alpha}dx$, the others being equal. We have%
\[
\int u_{\alpha\beta}u_{\beta i}\partial_{i}v^{\alpha}dx=-\int v^{\alpha
}\partial_{i}\left(  u_{\alpha\beta}u_{\beta i}\right)  dx.
\]
From Lemma \ref{lemma interp} (for instance with $f=u_{\alpha\beta}$,
$g=v^{\alpha}$, $h=u_{\beta i}$), we deduce
\[
\int u_{\alpha\beta}u_{\beta i}\partial_{i}v^{\alpha}dx\leq\epsilon\left\Vert
u\right\Vert _{W^{1,2}\left(  \mathbb{R}^{d}\right)  }^{2}+C_{\epsilon
}\left\Vert v\right\Vert _{L^{p}\left(  \mathbb{R}^{d}\right)  }^{\frac
{2p}{p-d}}\left\Vert u\right\Vert _{L^{2}\left(  \mathbb{R}^{d}\right)  }%
^{2}.
\]
Choosing $\epsilon>0$ small enough, we get%
\[
\frac{1}{2}\frac{d}{dt}\sum_{\alpha,\beta=1}^{d}\int u_{\alpha\beta}%
^{2}dx+\frac{\nu}{8}\sum_{\alpha,\beta=1}^{d}\int\left\vert \nabla
u_{\alpha\beta}\right\vert ^{2}dx\leq C\sum_{\alpha,\beta=1}^{d}\int
u_{\alpha\beta}^{2}dx
\]
where $C$ depends also on $\left\Vert v\right\Vert _{L^{\infty}\left(
0,T;L^{p}\left(  \mathbb{R}^{d}\right)  \right)  }$.

For the sake of clarity, let us restore now the notations $B^{\varepsilon}$
and $v^{\varepsilon}$, different from $B$ and $v$. We continue to denote by
$C>0$ a constant which depends only on $\left\Vert v\right\Vert _{L^{\infty
}\left(  0,T;L^{p}\left(  \mathbb{R}^{d}\right)  \right)  }$ and not on
$\varepsilon$. The previous identity, by Gronwall lemma and some other
elementary computations, gives us%
\begin{align*}
\sup_{t\in\left[  0,T\right]  }\left[  \int\left(  \mathbb{E}\left[
\left\vert B^{\varepsilon}\left(  t,x\right)  \right\vert ^{2}\right]
\right)  ^{2}dx\right]   &  \leq C\int\left(  \mathbb{E}\left[  \left\vert
B_{0}^{\varepsilon}\left(  x\right)  \right\vert ^{2}\right]  \right)
^{2}dx\\
&  =C\int\left\vert B_{0}^{\varepsilon}\left(  x\right)  \right\vert
^{4}dx\leq C\int\left\vert B_{0}\left(  x\right)  \right\vert ^{4}dx
\end{align*}
(here we see the need to assume $B_{0}\in L^{4}\left(  \mathbb{R}%
^{d},\mathbb{R}^{d}\right)  $). Therefore, given any $R>0$, we have%
\begin{equation}
\int_{0}^{T}\int_{B_{R}}\mathbb{E}\left[  \left\vert B^{\varepsilon}\left(
t,x\right)  \right\vert ^{2}\right]  dx\leq TR\left(  \int_{0}^{T}\int_{B_{R}%
}\left(  \mathbb{E}\left[  \left\vert B^{\varepsilon}\left(  t,x\right)
\right\vert ^{2}\right]  \right)  ^{2}dx\right)  ^{\frac{1}{2}}\leq CRT
\label{first basic est}%
\end{equation}
where the constant $C>0$ depends only on $\left\Vert v\right\Vert _{L^{\infty
}\left(  0,T;L^{p}\left(  \mathbb{R}^{d}\right)  \right)  }$ and
$\int\left\vert B_{0}\left(  x\right)  \right\vert ^{4}dx$.

\subsection{Estimates on $E\left[  B^{\varepsilon}e_{f}\right]  $}

Since $B^{\varepsilon}$ is regular, by It\^{o} calculus we can prove that
$V^{\varepsilon}(t,x):=\mathbb{E}[B^{\varepsilon}(t,x)e_{f}\left(  t\right)
]$ satisfies equation (\ref{PDE}), namely (we often omit again in the first
part of the section the superscript $\varepsilon$)%
\[
\partial_{t}V+\left(  v+h\right)  \cdot\nabla V=\mathcal{L}V+V\cdot
\nabla\left(  v+h\right)  .
\]
Hence, with $V_{0}^{\varepsilon}(x):=\mathbb{E}[B_{0}^{\varepsilon}%
(x)e_{f}\left(  0\right)  ]=B_{0}^{\varepsilon}(x)$,
\begin{align*}
&  \int_{\mathbb{R}^{d}}\left\vert V(t,x)\right\vert ^{2}\,dx-\int_{0}^{t}%
\int_{\mathbb{R}^{d}}V\cdot\mathcal{L}V\,dxds\\
&  =\int_{\mathbb{R}^{d}}\left\vert V_{0}(x)\right\vert ^{2}\,dx+\sum
_{i,j=1}^{d}\int_{0}^{t}\int_{\mathbb{R}^{d}}V^{i}\partial_{i}\left(
v+h\right)  ^{j}V^{j}\,dxds\\
&  =\int_{\mathbb{R}^{d}}\left\vert B_{0}(x)\right\vert ^{2}\,dx-\sum
_{i,j=1}^{d}\int_{0}^{t}\int_{\mathbb{R}^{d}}\left(  v+h\right)  ^{j}%
\partial_{i}\left(  V^{i}V^{j}\right)  \,dxds.
\end{align*}

We observe that%

\[
\sum_{i,j=1}^{d}\int_{0}^{t}\int_{\mathbb{R}^{d}}h^{j}\partial_{i}\left(
V^{i}V^{j}\right)  \,dxds.
\]

\[
\leq C\int_{0}^{t}|f|\int_{\mathbb{R}^{d}}\left\vert V(t,x)\right\vert
^{2}\,dx\ ds.
\]

From Proposition \ref{Prop ellipticity} and Lemma \ref{lemma interp} (for
instance with $f=V^{i}$, $g=v^{j}$, $h=V^{j}$), we get
\begin{align*}
&  \int_{\mathbb{R}^{d}}V^{2}(t,x)\,dx+\nu\int_{0}^{t}\int_{\mathbb{R}^{d}%
}\left\vert \nabla V(s,x)\right\vert ^{2}\,dxds\\
&  \leq\int_{\mathbb{R}^{d}}V_{0}^{2}(x)\,dx+\int_{0}^{t}(|f|+C)\int%
_{\mathbb{R}^{d}}\left\vert V(t,x)\right\vert ^{2}\,dx\ ds\\
&  +\epsilon\int_{0}^{t}\left\Vert V\right\Vert _{W^{1,2}\left(
\mathbb{R}^{d}\right)  }^{2}ds+C_{\epsilon}\sup_{t\in\left[  0,T\right]
}\left\Vert v\left(  t\right)  \right\Vert _{L^{p}\left(  \mathbb{R}%
^{d}\right)  }^{\frac{2p}{p-d}}\int_{0}^{t}\left\Vert V\right\Vert
_{L^{2}\left(  \mathbb{R}^{d}\right)  }^{2}ds.
\end{align*}
When $\epsilon$ is small enough, by Gronwall lemma and the inequality itself
we deduce (here we restore the superscript $\varepsilon$)%
\begin{equation}
\sup_{t\in\left[  0,T\right]  }\int_{\mathbb{R}^{d}}V^{\varepsilon}%
(t,x)^{2}\,dx+\int_{0}^{T}\left\Vert V^{\varepsilon}\left(  s,\cdot\right)
\right\Vert _{W^{1,2}\left(  \mathbb{R}^{d}\right)  }^{2}ds\leq C \label{esti}%
\end{equation}
where again (given $h$) the constant $C>0$ depends only on $\left\Vert
v\right\Vert _{L^{\infty}\left(  0,T;L^{p}\left(  \mathbb{R}^{d}\right)
\right)  }$ and $\int\left\vert B_{0}\left(  x\right)  \right\vert ^{4}dx$.

\subsection{Passage to the limit}

From the bound (\ref{first basic est}) and a diagonal procedure, we may
construct a sequence $B^{\varepsilon_{n}}$ which converges weakly to a
progressively measurable process $B$ in $L^{2}\left(  \left[  0,T\right]
\times B\left(  0,R\right)  \times\Omega;\mathbb{R}^{d}\right)  $ for every
$R>0$. Since $B^{\varepsilon}$ is a solution of (\ref{Adve}), it is also a
weak solution. The equation is linear and thus, over compact support test
functions, we may pass to the limit by means of the previous weak convergence
property; we apply the classical argument of \cite[Sect. II, Chapter
3]{Pardoux}, see also \cite[Theorem 15]{FGP2}. This proves that property (iii)
in the definition of quasiregular weak solution is satisfied;\ similarly, one
proves (i), using the definition on smooth compact support test functions
$\varphi$
\[
\operatorname{div}B\left(  \omega,t\right)  \left(  \varphi\right)  =-\int
B\left(  \omega,t,x\right)  \cdot\nabla\varphi\left(  x\right)  dx
\]
on test functions $\varphi\in C_{c}^{\infty}(\mathbb{R}^{d},\mathbb{R}^{d})$
and the argument of \cite[Sect. II, Chapter 3]{Pardoux}. Similarly, one checks
that the real valued process $\int B(t,x)\cdot\varphi(x)dx$ is the weak limit
in $L^{2}\left(  \left[  0,T\right]  \times\Omega\right)  $ of the subsequence
$\int B^{\epsilon_{n}}(t,x)\cdot\varphi(x)dx$\ and therefore it is
progressively measurable; from the identity of property (iii)\ it follows that
it is also an $\mathcal{F}_{t}$-semimartingale and it has a continuous modification.

From (\ref{esti}) there exists a subsequence $\varepsilon_{n}$, which can be
extracted from the subsequence used in the previous step, such that
$V^{\varepsilon_{n}}(t,x)$ converges weakly star to the function
$V(t,x)=\mathbb{E}[B(t,x)\,e_{f}]$ in $C([0,T];L^{2}(\mathbb{R}^{d}%
,\mathbb{R}^{d}))$ and such that $\nabla V^{\varepsilon_{n}}(t,x)$ converges
weakly to $\nabla V(t,x)$ in $L^{2}([0,T]\times\mathbb{R}^{d};\mathbb{R}^{d}%
)$. This allows us to conclude that $V\in L^{2}([0,T];H^{1}(\mathbb{R}%
^{d},\mathbb{R}^{d}))\cap C([0,T];L^{2}(\mathbb{R}^{d},\mathbb{R}^{d}))$ and,
again thanks to the linearity of the equations, to show that $V$ solves the
PDE (\ref{PDE}), namely property (iv) in Definition \ref{defisolu}. This
proves that $B$ a quasiregular weak solution.

\subsection{Extra regularity in the case of finite dimensional noise}

Consider the special case when $\sigma_{k}\left(  x\right)  =e_{k}$ for
$k=1,...,d$, where $e_{1},...,e_{d}$ is a basis of $\mathbb{R}^{d}$ and
$\sigma_{k}\left(  x\right)  =0$ for $k\geq4$. The equation of
characteristics, for the regularized field $v^{\varepsilon}$, is simply
\[
dX_{t}^{\epsilon}=v^{\varepsilon}(t,X_{t})\,dt+dW_{t}\,,\hspace{1cm}X_{0}=x\,
\]
where $W_{t}=\left(  W_{t}^{1},...,W_{t}^{d}\right)  $. Recall we have the
representation formula
\begin{equation}
B^{\varepsilon}(t,x)=(D\Phi_{t}^{\varepsilon})(\Psi_{t}^{\varepsilon}\left(
x\right)  )B_{0}^{\varepsilon}\left(  \Psi_{t}^{\varepsilon}\left(  x\right)
\right)  \label{repr formula2}%
\end{equation}
in terms of the (regularized) initial condition and the direct and inverse
flows $\Phi_{t}^{\varepsilon}$ and $\Psi_{t}^{\varepsilon}$ associated to this
equation. By Lemma 5 of \cite{Fre1} we have that for every $p\geq1$, there
exists $C_{p,T}>0$ such that
\begin{equation}
\sup_{t\in\lbrack0,T]}\sup_{x\in\mathbb{R}^{3}}\mathbb{E}[|D\Phi
_{t}^{\varepsilon}\left(  x\right)  |^{p}]\leq C_{p,T},\quad\text{uniformly in
$\epsilon>0$}. \label{estima2}%
\end{equation}
Since $\Phi_{t}^{\varepsilon}(\Psi_{t}^{\varepsilon})=Id_{x}$ we have%

\[
(D\Phi_{t}^{\varepsilon})(\Psi_{t}^{\varepsilon})=\left(  D\Psi_{t}%
^{\varepsilon}\right)  ^{-1}.
\]
We observe that $\left(  D\Psi_{t}^{\varepsilon}\right)  ^{-1}$ is equal to%

\[
\frac{1}{\det(D\Phi_{t}^{\varepsilon})}Cof(D\Phi_{t}^{\varepsilon})^{T}%
\]
where $Cof$ denoted the cofactor matrix of $D\Phi_{t}^{\varepsilon}$. By the
solenoidal hypothesis on $v$ we have%

\[
Det(D\Phi_{t}^{\varepsilon})=1
\]
and by inequality (\ref{estima2}) we deduce that $Cof(D\Phi_{t}^{\varepsilon
})^{T}\in L^{\infty}\left(  [0,T]\times\mathbb{R}^{3},L^{2}(\Omega)\right)  $.
Then, under the assumption that $B_{0}$ is bounded, $B^{\varepsilon}(t,x)$ is
uniformly bounded in $L^{2}\left(  \Omega\times\lbrack0,T]\times\mathbb{R}%
^{3}\right)  $ and in $L^{\infty}\left(  \mathbb{R}^{3}\times\lbrack
0,T],L^{2}(\Omega)\right)  $. Arguing as above on weakly star converging
subsequences, this allows to prove the existence of a quasiregular solution
with the additional property%
\[
B\in L^{\infty}\left(  \mathbb{R}^{3}\times\lbrack0,T],L^{2}(\Omega)\right)
.
\]

\section{Uniqueness\label{sect unique}}

\begin{proof}
\textit{Step 0. Set of solutions.} Remark that the set of quasiregular weak
solutions is a linear subspace of $L^{2}\left(  \Omega\times\lbrack
0,T]\times\mathbb{R}^{3}\right)  $, because the stochastic advection equation
is linear, and the regularity conditions is a linear constraint. Therefore, it
is enough to show that a quasiregular weak solution $B$ with initial condition
$B_{0}=0$ vanishes identically.

\textit{Step 1. $V=0$.} Let $V\left(  t,x\right)  =\mathbb{E}\left[  B\left(
t,x\right)  e_{f}\left(  t\right)  \right]  $, with $f\in L^{2}%
([0,T],\mathbb{R}^{n})\cap L^{\infty}([0,T],\mathbb{R}^{n})$. If we prove that
$V=0$, for arbitrary $f$, by Lemma \ref{lemma Wiener space} we deduce $B=0$.
The function $V$ satisfies
\[
\partial_{t}V+\left[  v+h,V\right]  =\mathcal{L}V
\]
with initial condition $V_{0}=0$, where $h\left(  t,x\right)  :=\sum_{k=1}%
^{n}f_{k}\left(  t\right)  \sigma_{k}\left(  x\right)  $. It is thus
sufficient to prove that a solution $V$ (in weak sense) of class
$L^{2}([0,T];H^{1}(\mathbb{R}^{3}))\cap C([0,T];L^{2}(\mathbb{R}^{3}))$ of
this equation, such that $V_{0}=0$, is identically equal to zero. Let us see
that this is a classical result of the variational theory of evolution equations.

Let $\mathcal{V}\subset\mathcal{H}\subset\mathcal{V}^{\prime}$ be the Gelfand
triple defined by
\begin{align*}
\mathcal{H}  &  =L_{\sigma}^{2}\left(  \mathbb{R}^{3},\mathbb{R}^{3}\right) \\
\mathcal{V}  &  =H_{\sigma}^{1}\left(  \mathbb{R}^{3},\mathbb{R}^{3}\right)
\end{align*}
where the subscript $\sigma$ denotes the fact that we take these vector fields
with divergence equal to zero. The norm $\left\vert .\right\vert
_{\mathcal{H}}$ and scalar product $\left\langle .,.\right\rangle
_{\mathcal{H}}$ are the usual ones, and the norm $\left\Vert .\right\Vert
_{\mathcal{V}}$ in $\mathcal{V}$ is defined by%
\[
\left\Vert f\right\Vert _{\mathcal{V}}^{2}=\sum_{i=1}^{3}\int_{\mathbb{R}^{3}%
}\left\vert \nabla f^{i}\left(  x\right)  \right\vert ^{2}dx+\int%
_{\mathbb{R}^{3}}\left\vert f\left(  x\right)  \right\vert ^{2}dx.
\]
Let $a:\left[  0,T\right]  \times\mathcal{V}\times\mathcal{V}\rightarrow
\mathbb{R}$ be the bilinear form defined on smooth fields $f,g$ as%
\[
a\left(  t,f,g\right)  =-\int_{\mathbb{R}^{3}}\mathcal{L}f\left(  x\right)
\cdot g\left(  x\right)  dx+\int_{\mathbb{R}^{3}}\left[  v+h,f\right]  \left(
x\right)  \cdot g\left(  x\right)  dx
\]
and extended to $\mathcal{V}\times\mathcal{V}$ by one integration by parts of
the second order term in $\mathcal{L}$; moreover, since $v$ is not
differentiable, we have to interpret also one term in $\int_{\mathbb{R}^{3}%
}\left[  v+h,f\right]  \left(  x\right)  \cdot g\left(  x\right)  dx$ by
integration by parts. More precisely,
\begin{align*}
a\left(  t,f,g\right)   &  =\sum_{i,j,\alpha=1}^{3}\int_{\mathbb{R}^{3}}%
a_{ij}\left(  x\right)  \partial_{j}f^{\alpha}\left(  x\right)  \partial
_{i}g^{\alpha}\left(  x\right)  dx+\sum_{i,j,\alpha=1}^{3}\int_{\mathbb{R}%
^{3}}g^{\alpha}\left(  x\right)  \partial_{j}f^{\alpha}\left(  x\right)
\partial_{i}a_{ij}\left(  x\right)  dx\\
&  -\sum_{i,\alpha,\beta=1}^{3}\int_{\mathbb{R}^{3}}b_{i}^{\alpha\beta}\left(
x\right)  \partial_{i}f^{\beta}\left(  x\right)  g^{\alpha}\left(  x\right)
dx-\sum_{\alpha,\beta=1}^{3}\int_{\mathbb{R}^{3}}c^{\alpha\beta}\left(
x\right)  f^{\beta}\left(  x\right)  g^{\alpha}\left(  x\right)  dx\\
&  +\sum_{\alpha,\beta=1}^{3}\int_{\mathbb{R}^{3}}\left(  v^{\alpha}\left(
t,x\right)  +h^{\alpha}\left(  t,x\right)  \right)  \partial_{\alpha}f^{\beta
}\left(  x\right)  g^{\beta}\left(  x\right)  dx\\
&  +\sum_{\alpha,\beta=1}^{3}\int_{\mathbb{R}^{3}}\left(  v^{\beta}\left(
t,x\right)  +h^{\beta}\left(  t,x\right)  \right)  \partial_{\alpha}\left(
f^{\alpha}\left(  x\right)  g^{\beta}\left(  x\right)  \right)  dx
\end{align*}
where we recall that $\partial_{i}a_{ij}$ is bounded continuous. Then the weak
form of equation $\partial_{t}V+\left[  v+h,V\right]  =\mathcal{L}V$, with
$V_{0}=0$, is equivalent to
\begin{equation}
\left\langle V\left(  t\right)  ,\phi\right\rangle _{\mathcal{H}}+\int_{0}%
^{t}a\left(  s,V\left(  s\right)  ,\phi\right)  ds=0. \label{abstract PDE}%
\end{equation}
for all $\phi\in\mathcal{V}$. Uniqueness for equations (\ref{PDE}) and
(\ref{abstract PDE}) are equivalent, in the class $V\in L^{2}\left(
0,T;\mathcal{V}\right)  \cap C\left(  \left[  0,T\right]  ;\mathcal{H}\right)
$. It is known, see \cite{Lions}, that uniqueness (and existence) in this
class holds when $a$ is measurable in the three variables, continuous and
coercive in the last two variables, namely%
\begin{equation}
\left\vert a\left(  t,f,g\right)  \right\vert \leq C\left\Vert f\right\Vert
_{\mathcal{V}}\left\Vert g\right\Vert _{\mathcal{V}} \label{continuity}%
\end{equation}%
\begin{equation}
a\left(  t,f,f\right)  \geq\nu\left\Vert f\right\Vert _{\mathcal{V}}%
^{2}-\lambda\left\vert f\right\vert _{\mathcal{H}}^{2} \label{coercivity}%
\end{equation}
for some constants $C,\lambda\geq0$, $\nu>0$, for a.e. $t$ and all
$f,g\in\mathcal{V}$. Let us prove these two properties. It is sufficient to
check them on the subset of smooth compact support divergence free fields
$f,g$.

Let us prove (\ref{continuity}). The first four terms in the explicit
expression for $a\left(  t,f,g\right)  $ can be bounded above by $C\left\Vert
f\right\Vert _{\mathcal{V}}\left\Vert g\right\Vert _{\mathcal{V}}$ because
$a_{ij},\partial_{i}a_{ij},b_{i}^{\alpha\beta},c^{\alpha\beta}$ are bounded.
The difficult terms are the last two. Again, since $h$ is bounded, the terms%
\[
\sum_{\alpha,\beta=1}^{3}\int_{\mathbb{R}^{3}}h^{\alpha}\left(  t,x\right)
\partial_{\alpha}f^{\beta}\left(  x\right)  g^{\beta}\left(  x\right)
dx+\sum_{\alpha,\beta=1}^{3}\int_{\mathbb{R}^{3}}h^{\beta}\left(  t,x\right)
\partial_{\alpha}\left(  f^{\alpha}\left(  x\right)  g^{\beta}\left(
x\right)  \right)  dx
\]
can be bounded above by $C\left\Vert f\right\Vert _{\mathcal{V}}\left\Vert
g\right\Vert _{\mathcal{V}}$. It remains to bound%
\[
\sum_{\alpha,\beta=1}^{3}\int_{\mathbb{R}^{3}}v^{\alpha}\left(  t,x\right)
\partial_{\alpha}f^{\beta}\left(  x\right)  g^{\beta}\left(  x\right)
dx+\sum_{\alpha,\beta=1}^{3}\int_{\mathbb{R}^{3}}v^{\beta}\left(  t,x\right)
\partial_{\alpha}\left(  f^{\alpha}\left(  x\right)  g^{\beta}\left(
x\right)  \right)  dx.
\]
But here we use repeatedly the first claim of Lemma Lemma \ref{lemma interp}
and bound also these terms with $C\left\Vert f\right\Vert _{\mathcal{V}%
}\left\Vert g\right\Vert _{\mathcal{V}}$. We have proved (\ref{continuity}).

Finally, let us show property (\ref{coercivity}). From Proposition
\ref{Prop ellipticity}, the part of $a\left(  t,f,f\right)  $ related to
$-\int_{\mathbb{R}^{3}}\mathcal{L}f\left(  x\right)  \cdot f\left(  x\right)
dx$ is bounded below by%

\[
\nu\int_{\mathbb{R}^{3}}\left\vert \nabla f\left(  x\right)  \right\vert
^{2}dx-C\int_{\mathbb{R}^{3}}\left\vert f\left(  x\right)  \right\vert
^{2}dx.
\]
The remaining terms, namely
\begin{align}
&  \sum_{\alpha,\beta=1}^{3}\int_{\mathbb{R}^{3}}\left(  v^{\alpha}\left(
t,x\right)  +h^{\alpha}\left(  t,x\right)  \right)  \partial_{\alpha}f^{\beta
}\left(  x\right)  f^{\beta}\left(  x\right)  dx\label{to be bounded}\\
&  +\sum_{\alpha,\beta=1}^{3}\int_{\mathbb{R}^{3}}\left(  v^{\beta}\left(
t,x\right)  +h^{\beta}\left(  t,x\right)  \right)  \partial_{\alpha}\left(
f^{\alpha}\left(  x\right)  f^{\beta}\left(  x\right)  \right)  dx
\end{align}
are bounded above in absolute value by
\[
\frac{\nu}{2}\int_{\mathbb{R}^{3}}\left\vert \nabla f\left(  x\right)
\right\vert ^{2}dx+C\int_{\mathbb{R}^{3}}\left\vert f\left(  x\right)
\right\vert ^{2}dx
\]
because of Lemma \ref{lemma interp}, with a suitable choice of $\epsilon>0$.
This implies $a\left(  t,f,f\right)  \geq\frac{\nu}{2}\left\Vert f\right\Vert
_{\mathcal{V}}^{2}-C\left\vert f\right\vert _{\mathcal{H}}^{2}$.

\textit{Step 2. Conclusion. Until now we have proved that, for every }$f\in
L^{2}([0,T],\mathbb{R}^{n})\cap L^{\infty}([0,T],\mathbb{R}^{n})$, the
function $\left(  t,x\right)  \mapsto\mathbb{E}[B(t,x)e_{f}\left(  t\right)
]$ is the zero element of the space $L^{2}\left(  0,T;\mathcal{V}\right)  \cap
C\left(  \left[  0,T\right]  ;\mathcal{H}\right)  $. We have to deduce that
$B=0$.

Being $\left(  t,x\right)  \mapsto\mathbb{E}[B(t,x)e_{f}\left(  t\right)  ]$
the zero element of $C\left(  \left[  0,T\right]  ;\mathcal{H}\right)  $, we
know that for every $t\in\left[  0,T\right]  $ we have%
\[
\int_{\mathbb{R}^{3}}\mathbb{E}[B(t,x)e_{f}\left(  t\right)  ]g(x)\,dx=0
\]
for all $g\in C_{c}^{\infty}(\mathbb{R}^{3},\mathbb{R}^{3})$; and this holds
true for all $e_{f}\in\mathcal{D}$. By linearity of the integral and the
expected value we also have that
\begin{equation}
\int_{\mathbb{R}^{3}}\mathbb{E}\left[  B(t,x)\,Y\right]  g(x)\,dx=0
\label{eq uY}%
\end{equation}
for every random variable $Y$ which can be written as a linear combination of
a finite number of $e_{f}\left(  t\right)  $ and by density also the
restriction $f\in L^{\infty}([0,T],\mathbb{R}^{n})$ can be removed. Since by
Lemma \ref{lemma Wiener space} the span generated by $e_{f}\left(  t\right)  $
is dense in $L^{2}(\Omega,\overline{\mathcal{G}}_{t})$, \eqref{eq uY} holds
for any $Y\in L^{2}(\Omega,\overline{\mathcal{G}}_{t})$. Namely, we have%
\[
\mathbb{E}\left[  \int_{\mathbb{R}^{3}}B(t,x)g(x)dx\,Y\right]  =0
\]
for every $Y\in L^{2}(\Omega,\overline{\mathcal{G}}_{t})$. Since, by
assumption, $\int_{\mathbb{R}^{3}}B(t,x)g(x)dx$ is $\overline{\mathcal{G}}%
_{t}$-measurable, we deduce
\[
\int_{\mathbb{R}^{3}}B(t,x)g(x)dx=0.
\]
This holds true for every $g\in C_{c}^{\infty}(\mathbb{R}^{3},\mathbb{R}^{3}%
)$, hence $B(t,\cdot)=0$.
\end{proof}

\section*{Acknowledgements}


The authors express their gratitude to anonymous referees that helped to
improve considerably the structure of the work and the exposition. The work of
F. F. is supported in part by University of Pisa under the Project
PRA\_2016\_41. The work of C. O. is partially supported by FAPESP 2015/04723-2
and CNPq through the grant 460713/2014-0. This paper was written while the
second author visited the University of Pisa with the partial support of
GNAMPA, INDAM.


\end{document}